\documentclass[reqno]{amsart}
\usepackage{amsmath}
\usepackage{amsfonts}
\usepackage{amssymb}
\usepackage{color}
\usepackage{mathrsfs}
\usepackage{graphicx}
\usepackage{caption,cite}

\newtheorem{thm}{Theorem}[section]

\newtheorem{prop}[thm]{Proposition}
\newtheorem{lem}[thm]{Lemma}

\newtheorem{rmk}[thm]{Remark}

\def\D{\mathrm{D}}

\def\M{\mathscr{M}}
\def\O{\mathscr{O}}
\def\Z{\mathscr{Z}}

\def\d{\mathrm{d}}
\def\e{\mathrm{e}}
\def\h{\mathrm{h}}
\def\s{\mathrm{s}}

\def\Cset{\mathbb{C}}
\def\Nset{\mathbb{N}}
\def\Pset{\mathbb{P}}
\def\Rset{\mathbb{R}}
\def\Zset{\mathbb{Z}}

\def\Re{\mathrm{Re}}
\def\Span{\mathrm{span}}

\def\fix{\mathrm{Fix}}

\def\id{\mathrm{id}}
\def\loc{\mathrm{loc}}

\def\sech{\mathrm{sech}}

\def\half{{\textstyle\frac{1}{2}}}

\def\epsilon{\varepsilon}

\def\theequation{\arabic{section}.\arabic{equation}}

\makeatletter
 \@addtoreset{equation}{section}
 \def\widebar{\accentset{{\cc@style\underline{\mskip10mu}}}}
 \makeatother

\begin{document}

\title[Bifurcations of homoclinic orbits II: Reversible systems]{%
Analytic and algebraic conditions for bifurcations of homoclinic orbits II:
Reversible systems}

\author{Kazuyuki~Yagasaki}

\address{
Department of Applied Mathematics and Physics,
Graduate School of Informatics,
Kyoto University,
Yoshida-Honmachi, Sakyo-ku,
Kyoto 606-8501, Japan}
\email{yagasaki@amp.i.kyoto-u.ac.jp}

\date{\today}



\begin{abstract}
Following Part~I, we consider a class of reversible systems
 and study bifurcations of homoclinic orbits to hyperbolic saddle equilibria.
Here we concentrate on the case in which homoclinic orbits are symmetric,
 so that only one control parameter is enough to treat their bifurcations,
 as in Hamiltonian systems.
First, we modify and extend arguments of Part~I to show
 in a form applicable to general systems discussed there that
 if such bifurcations occur in four-dimensional systems,
 then variational equations around the homoclinic orbits are integrable
 in the meaning of differential Galois theory.
We next extend the Melnikov method of Part~I to reversible systems
 and obtain theorems on saddle-node, transcritical and pitchfork bifurcations
 of symmetric homoclinic orbits.
We illustrate our theory for a four-dimensional system,
 and demonstrate the theoretical results by numerical ones.
\end{abstract}

\maketitle

\section{Introduction}

In the companion paper \cite{BY12a}, which we refer to as Part I here, 
 we studied bifurcations of homoclinic orbits to hyperbolic saddle equilibria
 in a class of systems including Hamiltonian systems.
They also arise as bifurcations of solitons or pulses in partial differential equations (PDEs),
 and have attracted much attention even in the fields of PDEs and nonlinear waves
 (see, e.g., Section~2 of \cite{S02}).
Only one control parameter is enough to treat these bifurcations in Hamiltonian systems
 but two control parameters are needed in general.
We applied a version of Melnikov's method due to Gruendler \cite{G92}
 to obtain some theorems on saddle-node and pitchfork types of bifurcations
 for homoclinic orbits in systems of dimension four or more.
Furthermore, we proved that if these bifurcations occur in four-dimensional systems,
 then variational equations (VEs) around the homoclinic orbits are integrable
 in the meaning of differential Galois theory \cite{CH11,PS03}
 when there exist analytic invariant manifolds on which the homoclinic orbits lie.
In \cite{YY20}, spectral stability of solitary waves,
 which correspond to such homocinic orbits in a two-degree-of-freedom Hamiltonian system,
 in coupled nonlinear Schr\"odinger equations were also studied.
 
In this part, we consider reversible systems of the form
\begin{equation}
\dot{x}=f(x;\mu),\quad
(x,\mu)\in\Rset^{2n}\times\Rset,
\label{eqn:sys}
\end{equation}
where $f:\Rset^{2n}\times\Rset\rightarrow\Rset^{2n}$ is analytic,
 $\mu$ is a parameter and $n\ge 2$ is an integer,
 and continue to discuss bifurcations of homoclinic orbits to hyperbolic saddles.
A rough sketch of the results were briefly stated in \cite{Y15}.
Our precise assumptions are as follows:
\begin{enumerate}
\setlength{\leftskip}{-0.5em}
\item[\bf(R1)]
The system \eqref{eqn:sys} is \emph{reversible},
 i.e., there exists a linear involution $R$
 such that $f(Rx;\mu)+Rf(x;\mu)=0$ for any $(x,\mu)\in\Rset^{2n}\times\Rset$.
Moreover, $\dim\fix(R)$ $=n$, where $\fix(R)=\{x\in\Rset^{2n}\,|\,Rx=x\}$.
\item[\bf(R2)]
The origin $x=0$, denoted by $O$,
 is an equilibrium in \eqref{eqn:sys} for all $\mu\in\Rset$, i.e., $f(0;\mu)=0$.
\end{enumerate}
Note that $O\in\fix(R)$ since $RO=O$.
A fundamental characteristic of reversible systems is that
 if $x(t)$ is a solution, then so is $Rx(-t)$.
We call a solution (and the corresponding orbit) \emph{symmetric}
 if $x(t)=Rx(-t)$.
It is a well-known fact that an orbit is symmetric
 if and only if it intersects the space $\fix(R)$ \cite{VF92}. 
Moreover, if $\lambda\in\Cset$ is an eigenvalue of $\D_x f(0;\mu)$,
 then so are $-\lambda$ and $\overline{\lambda}$,
 where the overline 
 represents the complex conjugate. 
See also \cite{LR98} for general properties of reversible systems.
We also assume the following.
\begin{enumerate}
\setlength{\leftskip}{-0.5em}
\item[\bf(R3)]
The Jacobian matrix $\D_x f(0;0)$ has $2n$ eigenvalues
 $\pm\lambda_1,\ldots,\pm\lambda_n$ such that $0<\Re\lambda_1\le\cdots\le\Re\lambda_n$
 (i.e., the origin is a hyperbolic saddle).
\item[\bf(R4)]
The equilibrium $x=0$ has a symmetric homoclinic orbit $x^\h(t)$ with $x^\h(0)\in\fix(R)$ at $\mu=0$.
{\color{black}Let $\Gamma_0=\{x^\h(t)|t\in\Rset\}\cup\{0\}$.}
\end{enumerate}

Assumptions similar to (R3) and (R4) were made
 in (M1) and (M2) for general multi-dimensional systems 
 and in (A1) and (A2) for four-dimensional systems in Part~I{\color{black}:
\begin{enumerate}
\setlength{\leftskip}{-0.5em}
\item[\bf(A1)]
The origin $x=0$ is a hyperbolic saddle equilibrium (in \eqref{eqn:sys} with $n=2$) at $\mu=0$,
 such that $\D_x f(0; 0)$ has four real eigenvalues,
 $\tilde{\lambda}_1\le\tilde{\lambda}_2< 0<\tilde{\lambda}_3\le\tilde{\lambda}_4$.
\item[\bf(A2)]
At $\mu=0$ the hyperbolic saddle $x=0$ has a homoclinic orbit $x^\h(t)$.
Moreover, there exists a two-dimensional analytic invariant manifold $\M$
 containing $x=0$ and $x^\h(t)$.
\end{enumerate}
In particular, in (R1)-(R4), we do not assume that
 there exists such an invariant manifold as $\M$ in (A2).
It follows from (R3) that the origin $x=0$ is still a hyperbolic saddle near $\mu=0$
 under some change of coordinates if necessary, as in (R2).}

Reversible systems are frequently encountered in applications
 such as mechanics, fluids and optics,
 and have attracted much attention in the literature \cite{LR98}.
One of {\color{black}the} characteristic properties of reversible systems
 is that homoclinic orbits to hyperbolic saddles  are typically symmetric
 and continue to exist when their parameters are varied if so,
 in contrast to the fact that such orbits do not persist in general systems.
In \cite{K97} saddle-node bifurcations of homoclinic orbits to hyperbolic saddles
 in reversible systems 
 were previously discussed and shown to be codimension-one or -two
 depending on whether the homoclinic orbits are symmetric or not.
Here we concentrate on the case in which homoclinic orbits are symmetric,
 so that only one control parameter is enough to treat their bifurcations,
 as in Hamiltonian systems {\color{black}(see Part~I)}.
For asymmetric homoclinic orbits,
 the arguments of Part~I for non-Hamiltonian systems can apply.

The object of this paper is twofold.
First, we consider the case of $n=2$,
 and modify and extend arguments given in Part~I to show that
 if a bifurcation of the homoclinic orbit $x^\h(t)$ occurs at $\mu=0$,
 then the VE of \eqref{eqn:sys} around $x^\h(t)$ at $\mu=0$,
\begin{equation}
\dot{\xi}=\D_xf(x^\h(t);0)\xi,\quad
{\color{black}\xi\in\Cset^4,}
\label{eqn:ve}
\end{equation}
is integrable in the meaning of differential Galois theory under some conditions
 {\color{black}even if there does not exist such an invariant manifold as $\M$ in (A2).
Here the domain on which
 Eq.~\eqref{eqn:ve} is defined has been extended to a neighborhood of $\Rset$ in $\Cset$.
Such an extension is possible since $f(x;\mu)$ and $\D_xf(x;\mu)$ are analytic.}
We assume the following {\color{black}three} conditions:
\begin{enumerate}
\setlength{\leftskip}{-0.5em}
{\color{black}
\item[\bf(B1)]
The origin $x=0$ is a hyperbolic saddle equilibrium and has a homoclinic orbit $x^\h(t)$
 in \eqref{eqn:sys} with $n=2$ at $\mu=0$, such that $\D_x f(0; 0)$ has four real eigenvalues,
 $\tilde{\lambda}_1\le\tilde{\lambda}_2< 0<\tilde{\lambda}_3\le\tilde{\lambda}_4$.}
\item[\bf(B2)]
The homoclinic orbit $x^\h(t)$ is expressed as
\begin{equation}
x^\h(t)=
\begin{cases}
h_+(e^{\lambda_- t}) & \mbox{for $\Re\,t>0$};\\
h_-(e^{\lambda_+ t}) & \mbox{for $\Re\,t<0$},
\end{cases}
\label{eqn:b1}
\end{equation}
in a neighborhood $U$ of {\color{black}$t=0$ in $\Cset$,
 where $h_\pm:U\to\Cset^4$ are certain analytic functions
 with their derivatives satisfying} $h_\pm'(0)\neq 0$,
{\color{black}$\lambda_+=\tilde{\lambda}_3$ or $\tilde{\lambda}_4$
 and $\lambda_-=\tilde{\lambda}_1$ or $\tilde{\lambda}_2$.
When the system~\eqref{eqn:sys} is reversible and $x^\h(t)$ is symmetric
 as in (R1)-(R4),
 we have $\lambda_\pm=\mp\tilde{\lambda}_j$ for $j=1$ or $2$.}
\item[\bf(B3)]
The VE \eqref{eqn:ve} has a solution $\xi=\varphi(t)$ such that
\begin{equation}
\begin{split}
&\qquad
\varphi(\lambda_-^{-1}\log z)
=a_+(z)z^{\lambda_+'/\lambda_-}+b_{1+}(z)z^{\tilde{\lambda}_1/\lambda_-}
 +b_{2+}(z)z^{\tilde{\lambda}_2/\lambda_-},\\
&\qquad\mbox{or}\quad
\varphi(\lambda_-^{-1}\log z)
=a_+(z)z^{\lambda_+'/\lambda_-}+b_{1+}(z)z+b_{2+}(z)z\log z
\end{split}
\label{eqn:b2a}
\end{equation}
and
\begin{equation}
\begin{split}
&
\varphi(\lambda_+^{-1}\log z)
=a_-(z)z^{\lambda_-'/\lambda_+}+b_{1-}(z)z^{\tilde{\lambda}_3/\lambda_+}
 +b_{2-}(z)z^{\tilde{\lambda}_4/\lambda_+},\\
&\mbox{or}\quad
\varphi(\lambda_+^{-1}\log z)
=a_-(z)z^{\lambda_-'/\lambda_+}+b_{1-}(z)z+b_{2-}(z)z\log z
\end{split}
\label{eqn:b2b}
\end{equation}
{\color{black}in $|z|\ll 1$}, where $a_{\pm}(z)$ and $b_{j\pm}(z)$, $j=1,2$,
 are certain analytic functions in $U$ with $a_\pm(0)\neq 0$
 and $\lambda_+'=\tilde{\lambda}_3$ or $\tilde{\lambda}_4$
 and $\lambda_-'=\tilde{\lambda}_1$ or $\tilde{\lambda}_2$.
\end{enumerate}
{\color{black}Note that if assumptions~(A1) and (A2) hold, then so does (B1).}
In (B3), we have $\varphi(t)e^{-\lambda_+' t}=\O(1)$ as $t\to+\infty$
 and $\varphi(t)e^{\lambda_-' t}=\O(1)$ as $t\to-\infty$.
Moreover, {\color{black}the second equations in \eqref{eqn:b2a} and \eqref{eqn:b2b} hold
 only if $\tilde{\lambda}_1=\tilde{\lambda}_2$ and $\tilde{\lambda}_3=\tilde{\lambda}_4$,
 respectively.
Note that the existence of such an invariant manifold as $\M$ in (A2) is not assumed.}
We prove the following theorem.

\begin{thm}
\label{thm:main}
Let $n=2$ and suppose that the following condition holds
 along with {\color{black}\rm(B1)-(B3):}
\begin{enumerate}
\setlength{\leftskip}{-1em}
\item[\bf(C)]
The VE \eqref{eqn:ve} has another linearly independent bounded solution.
\end{enumerate}
Then the VE \eqref{eqn:ve} has a triangularizable differential Galois group,
 when regarded as a complex differential equation with meromorphic coefficients
 in a desingularized neighborhood $\Gamma_\loc$ of the homoclinic orbit $x^\h(t)$
 in $\Cset^4$.
\end{thm}

\begin{figure}
\begin{center}
\includegraphics[scale=0.8]{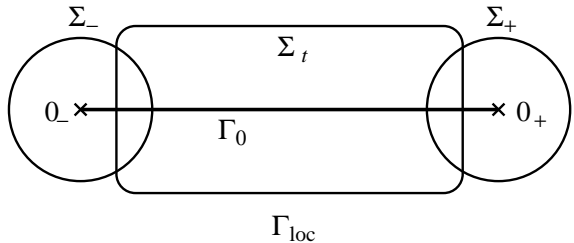}
\caption{Local Riemann surface $\Gamma_\loc$
 and its covering $\{\Sigma_\pm,\Sigma_t\}$.}
\label{fig:1a}
\end{center}
\end{figure}

Note that $\xi=\dot{x}^\h(t)$ is a bounded solution to \eqref{eqn:ve}. 
A proof of Theorem~\ref{thm:main} is given in Section~2.
Here $\Gamma_\loc$ is a local Riemann surface
 that is given by $\Sigma_-\cup\Sigma_t\cup\Sigma_+$
 and taken sufficiently narrowly,
 where $\Sigma_+$ and $\Sigma_-$ are, respectively, neighborhoods of $0_+$ and $0_-$,
  which are represented in the temporal parameterization of $x^\h(t)$
  by $t=+\infty$ and $-\infty$,
  and $\Sigma_t$ be a neighborhood of $\Gamma_0\setminus(\Sigma_+\cup\Sigma-)$
  (see Fig.~\ref{fig:1a}).
See Section~2 for more details.
This theorem means that under conditions (B1)-(B3)
 the VE \eqref{eqn:ve} is integrable in the meaning of differential Galois theory
 if condition~(C) holds.
The statement of Theorem~\ref{thm:main} holds for general systems,
 especially even if the homoclinic orbit is asymmetric in \eqref{eqn:sys}.
It was proved in Part~I, where instead of such conditions as (B2) and (B3)
 the existence of a two-dimensional invariant manifold containing $x^\h(t)$ was assumed.
Conditions~(B2) and (B3) with appropriate modifications also hold
 under the restricted assumption.
We give simple conditions guaranteeing (B2) and (B3) {\color{black}under condition~(B1)}
 in the general setting.
See Proposition~\ref{prop:2a}.
As stated below (see also Part~I),
 condition~(C) is a necessary and sufficient condition
 for the occurrence of some bifurcations of the homoclinic orbit
 under certain {\color{black}nondegenerate} conditions.

Secondly, {\color{black}we extend} the Melnikov method of Part~I to reversible systems
 and obtain some theorems on saddle-node, transcritical and pitchfork bifurcations
 of symmetric homoclinic orbits.
In particular, it is shown without the restriction of $n=2$
 that if and only if condition~(C) holds
 and no further linearly independent {\color{black}bounded} solution to the VE \eqref{eqn:ve} exists,
 then saddle-node, transcritical or pitchfork bifurcations of symmetric homoclinic orbits occur
 under some {\color{black}nondegenerate} conditions.
We emphasize that our result is not an immediate extension of Part~I:
The reversibility of the systems as well as their symmetry
 is well used to detect codimension-one bifurcations of symmetric homoclinic orbits.
So bifurcations of symmetric homoclinic orbits in reversible systems
 are proven to be of codimension-one, again. 
We also illustrate our theory for a four-dimensional system.
We perform numerical computations using the computer tool \texttt{AUTO} \cite{DO07},
  and {\color{black}demonstrate the usefulness and validity of the theoretical results
  comparing them with the numerical ones.}


The outline of the paper is as follows:
In Section~2 we give a proof of Theorem~\ref{thm:main}
 in a form applicable to general {\color{black}four-dimensional} systems discussed in Part~I.
In Section~3 we extend the Melnikov method
 and obtain analytic conditions for bifurcations of symmetric homoclinic orbits.
Finally, our theory is illustrated for the example
 along with numerical results in Section~4.

\section{Algebraic Condition}

In this section we restrict ourselves to the case of $n=2$
 and give a proof of Theorem~\ref{thm:main}
 in a form applicable to general systems stated above,
 {\color{black}i.e., without assumptions (R1)-(R4)}.
We first recall Lemma~2.1 of Part~I in our setting.

\begin{lem}
\label{lem:2a}
{\color{black}Under condition~{\rm(B1)}, there exist
 a fundamental matrix $\Phi(t)=(\varphi_1(t),\ldots,\varphi_4(t))$ of \eqref{eqn:ve}
 and a permutation $\sigma$ on four symbols $\{1,2,3,4\}$} such that
\begin{align*}
&
\varphi_j(t)t^{-k_j}e^{-\tilde{\lambda}_jt}=\O(1)
\quad\mbox{as $t\to+\infty$},\\
&
\varphi_j(t)t^{-k_\sigma(j)}e^{-\tilde{\lambda}_{\sigma(j)}t}=\O(1)
\quad\mbox{as $t\to-\infty$},
\end{align*}
where {\color{black}$k_j=0$ or $1$, $j=1$-$4$}.
\end{lem}

We now assume that the hypotheses of Theorem~\ref{thm:main} hold along with condition~(C)
 and that {\color{black}$\lambda_-=\tilde{\lambda}_1(<0)$
 and $\lambda_+=\tilde{\lambda}_4(>0)$ for simplicity
 since the other cases can be treated similarly.
We regard the VE \eqref{eqn:ve} as a differential equation
 defined on a neighborhood of $\Rset$ in $\Cset$, as stated in Section~1.}
From \eqref{eqn:b1} {\color{black}in (B3)} we easily obtain the following.

\begin{lem}
\label{lem:2b}
We have
\[
\varphi_1(t)=
\begin{cases}
\lambda_-e^{\lambda_-t}\,h_+'(e^{\lambda_- t}) & \mbox{for $\Re\,t>0$};\\
\lambda_+e^{\lambda_+t}\,h_-'(e^{\lambda_+ t}) & \mbox{for $\Re\,t<0$},
\end{cases}
\]
{\color{black} in $U$, where 
 $h_\pm'(z)$ represent the derivatives of $h_\pm(z)$}.
\end{lem}

Let $U_\Rset$ denote a neighborhood of $\Rset\cup\{\infty\}$
 in the Riemann sphere $\Pset^1=\Cset\cup\{\infty\}$ 
 and let $\Gamma=\{x^\h(t)\mid t\in U_\Rset\}$.
Introducing two points $0_+$ and $0_-$ for the double point $x=0$,
 we desingularize the curve $\Gamma_0=\{x=x^\h(t)\mid t\in\Rset\}\cup\{0\}$
 on $\Gamma$.
Here the points $0_+$ and $0_-$ are represented in the temporal parameterization
 by $t=+\infty$ and $-\infty$, respectively.
Let $\Sigma_\pm$ be neighborhoods of $0_\pm$ on $\Gamma$,
 and take a sufficiently narrow, simply connected neighborhood $\Sigma_t$
 of $\Gamma_0\setminus(\Sigma_+\cup\Sigma_-)$.
We set $\Gamma_\loc=\Sigma_-\cup\Sigma_t\cup\Sigma_+$.
See Fig.~\ref{fig:1a}.
Using the expression \eqref{eqn:b1}, 
 we introduce three charts on the Riemann surface $\Gamma_\loc$:
 a chart $\Sigma_+$ (resp. $\Sigma_-$) near $0_+$ (resp. near $0_-$)
 with coordinates $z=\e^{\lambda_- t}$ (resp. $z=\e^{\lambda_+ t}$),
 and a chart $\Sigma_t$ bridging them with a coordinate $t$.
The transformed VE becomes
\[
\frac{\d\xi}{\d z}
 =\frac{1}{\lambda_\mp z}\D_x f(h_\pm(z);0)\,\xi
\]
in the charts $\Sigma_{\pm}$
 while it has the same form in the chart $\Sigma_t$ as the original one.
We take a sufficiently small surface as $\Gamma_{\loc}$
 such that it contains no other singularity except $z=0$ in $\Sigma_{\pm}$.
We easily see that $\D_x f(h_\pm(z);0)$ are analytic in $\Sigma_\pm$
 to obtain the following lemma.

\begin{lem}
\label{lem:2c}
The singularities of the transformed VE for \eqref{eqn:ve} at $z=0$ in $\Sigma_{\pm}$
 are regular.
Thus, the transformed VE is Fuchsian on $\Gamma_{\loc}$.
\end{lem}


We are now in a position to prove Theorem~\ref{thm:main}.

\begin{proof}[Proof of Theorem~$\ref{thm:main}$]
We first see that
 $\Phi(\lambda_\mp^{-1}\log z)$ in $\Sigma_{\pm}$
 is a fundamental matrix of the transformed VE.
So the $4\times 4$ monodromy matrices $M_\pm$
 along small circles of radius $\epsilon>0$ centered at $z=0$ in $\Sigma_{\pm}$
 are computed as
\[
M_\pm=\Phi(\lambda_\mp^{-1}\log\epsilon)^{-1}\Phi(\lambda_\mp^{-1}(\log \epsilon+2\pi i)).
\]
{\color{black}
Let $e_j$ denote the vector of which the $j$-th component is one and the others are zero
 for $j=1$-$4$.
Since $\varphi_1(t)=\Phi(t)e_1$, it follows from Lemma~\ref{lem:2b} that}
\[
M_\pm e_1=e_1.
\]
To prove the theorem,
 we only have to show that $M_\pm$ are simultaneously triangularizable,
 since by Corollary 3.5 of Part~I the differential Galois group is triangularizable if so.

Suppose that $\tilde{\lambda}_1\neq\tilde{\lambda}_2$
 and $\tilde{\lambda}_3\neq\tilde{\lambda}_4$.
Then the Jacobian matrix $\D_xf(0;0)$ is diagonalizable,
 so that by Lemma~\ref{lem:2a} and condition~(C)
\begin{equation}
\begin{split}
&
\varphi_2(\lambda_-^{-1}\log z)=a_{1+}(z)z+a_{2+}(z)z^{\tilde{\lambda}_2/\lambda_-},\\
&
\varphi_2(\lambda_+^{-1}\log z)=a_{1-}(z)z+a_{2-}(z)z^{\tilde{\lambda}_3/\tilde{\lambda}_+}
\end{split}
\label{eqn:2b}
\end{equation}
{\color{black}for $|z|\ll 1$}, where $a_{j\pm}$, $j=1,2$, are certain analytic functions. 
Hence,
\begin{equation}
M_\pm e_2\in\Span\{e_1,e_2\}.
\label{eqn:2c}
\end{equation}
Moreover, {\color{black}since the first equations of \eqref{eqn:b2a} and \eqref{eqn:b2b}
 in (B3) hold,}
 there exists $v\in\Rset^4$ such that
\begin{equation}
M_\pm v\in\Span\{v,e_1,e_2\}.
\label{eqn:2d}
\end{equation}
Thus, $M_\pm$ are simultaneously triangularizable.

{\color{black}
Next assume that $\tilde{\lambda}_1=\tilde{\lambda}_2$
 but $\tilde{\lambda}_3\neq\tilde{\lambda}_4$.
If the eigenvalue $\tilde{\lambda}_1=\tilde{\lambda}_2$ is of geometric multiplicity two,
 then we can prove that $M_\pm$ are simultaneously triangularizable as in the above case.
So we assume that it is of geometric multiplicity one and algebraic multiplicity two.
Instead of the first equation of \eqref{eqn:2b} we have
\[
\varphi_2(\lambda_-^{-1}\log z)=a_{1+}(z)z+a_{2+}(z)z\log z
\]
in $|z|\ll 1$, so that Eq.~\eqref{eqn:2c} holds.
Moreover, even if not the first but second equation in \eqref{eqn:b2a} holds,
 there exists $v\in\Rset^4$ satisfying \eqref{eqn:2d} as above.
Hence, $M_\pm$ are simultaneously triangularizable.
Similarly, we can show that $M_\pm$ are simultaneously triangularizable
 when $\tilde{\lambda}_3=\tilde{\lambda}_4$ but $\tilde{\lambda}_1\neq\tilde{\lambda}_2$.

Finally, assume that $\tilde{\lambda}_1=\tilde{\lambda}_2$
 and $\tilde{\lambda}_3=\tilde{\lambda}_4$.
If the eigenvalue $\tilde{\lambda}_1=\tilde{\lambda}_2$
 and/or $\tilde{\lambda}_3=\tilde{\lambda}_4$ is of geometric multiplicity two,
 then we can prove that $M_\pm$ are simultaneously triangularizable as in the above two cases.
So we assume that they are of geometric multiplicity one and algebraic multiplicity two.
Instead of \eqref{eqn:2b} we have
\begin{align*}
&
\varphi_2(\lambda_-^{-1}\log z)=a_{1+}(z)z+a_{2+}(z)z\log z,\\
&
\varphi_2(\lambda_+^{-1}\log z)=a_{1-}(z)z+a_{2-}(z)z\log z
\end{align*}
in $|z|\ll 1$, so that Eq.~\eqref{eqn:2c} holds.
even if not the first but second equation in \eqref{eqn:b2a} holds,
 there exists $v\in\Rset^4$ satisfying \eqref{eqn:2d} as above.
Hence, $M_\pm$ are simultaneously triangularizable.}
\end{proof}


{\color{black}
At the end of this section we give simple conditions guaranteeing (B2) and (B3)
 under condition~(B1).}
 
\begin{prop}
\label{prop:2a}
{\color{black}
Under condition~{\rm(B1)}, conditions~{\rm(B2)} and {\rm(B3)} hold
 if one of the following conditions holds{\rm:}
\begin{enumerate}
\setlength{\leftskip}{-1.5em}
\item[(i)]
There exists a two-dimensional analytic invariant manifold $\M$
 containing $x = 0$ and $x^\h(t);$
\item[(ii)]
As well as $\sigma(3)=2$ and $k_2,k_3=0$,
 we can take $\varphi_1(t)=\dot{x}^\h(t)$ with $\sigma(1)=4$ and $k_1,k_4=0$
 or condition~{\rm(B2)} holds.
\end{enumerate}}
\end{prop}

\begin{proof}
First, assume condition~(i).
We recall from Part~I that the VE \eqref{eqn:ve} can be rewritten
 in certain coordinates $(\chi,\eta)\in\Rset^2\times\Rset^2$ as
\[
\begin{pmatrix}
\dot{\chi}\\
\dot{\eta}
\end{pmatrix}
=
\begin{pmatrix}
A_\chi(x^\h(t)) & A_c(x^\h(t))\\
0 & A_\eta(x^\h(t))
\end{pmatrix}
\begin{pmatrix}
\chi\\
\eta
\end{pmatrix},
\]
where $A_\chi(x),A_c(x),A_\eta(x)$ are analytic $2\times 2$ matrix functions of $x\in\Rset^4$,
 and that the tangent space of $\M$ is given by $\eta=0$.
See Section~4.1 (especially, Eq.~(32)) of Part~I.
Especially, $A_\chi(0)$ and $A_\eta(0)$ have positive and negative eigenvalues,
 say $\tilde{\lambda}_1<0<\tilde{\lambda}_3$ and $\tilde{\lambda}_2<0<\tilde{\lambda}_4$.
Hence, $\dot{x}^\h(t)$ corresponds to a solution to
\begin{equation}
\dot{\chi}=A_\chi(x^\h(t))\chi,
\label{eqn:prop2a}
\end{equation}
so that condition~(B2) holds.
Equation~\eqref{eqn:prop2a} also has a solution satisfying
\begin{align*}
&
\chi(t)e^{-\tilde{\lambda}_3 t}=\O(1)\quad\mbox{as $t\to+\infty$},\\
&
\chi(t)e^{-\tilde{\lambda}_1 t}=\O(1)\quad\mbox{as $t\to-\infty$},
\end{align*}
which mean that
\begin{align*}
&
\chi(\tilde{\lambda}_1^{-1}\log z)
 =a_+(z)z^{\tilde{\lambda}_3/\tilde{\lambda}_1}+b_+(z)z,\\
&
\chi(\tilde{\lambda}_3^{-1}\log z)
 =a_-(z)z^{\tilde{\lambda}_1/\tilde{\lambda}_3}+b_-(z)z
\end{align*}
{\color{black}for $|z|\ll 1$},
 where $a_{\pm}(z)$ and $b_{\pm}(z)$ are certain analytic functions with $a_\pm(0)\neq 0$.
This means condition~(B3).
Similar arguments can apply even if the eigenvalues of $A_\chi(0)$ and $A_\eta(0)$
 are not $\tilde{\lambda}_1,\tilde{\lambda}_3$ and $\tilde{\lambda}_2,\tilde{\lambda}_4$,
 respectively.

{\color{black}Next,} assume condition~(ii).
If condition~(B2) holds,
 then there exists a solution to the VE \eqref{eqn:ve} such that
\begin{align*}
&
\varphi_3(t)e^{-\tilde{\lambda}_3t}=\O(1)\quad\mbox{as $t\to+\infty$},\\
&
\varphi_3(t)e^{-\tilde{\lambda}_2t}=\O(1)\quad\mbox{as $t\to-\infty$},
\end{align*}
which mean condition~(B3).
On the other hand,
 if $\varphi_1(t)=\dot{x}^\h(t)$ with $\sigma(1)=4$ and $k_1,k_4=0$, then
\begin{align*}
&
\dot{x}^\h(t)e^{-\tilde{\lambda}_1t}=\O(1)\quad\mbox{as $t\to+\infty$},\\
&
\dot{x}^\h(t)e^{-\tilde{\lambda}_4t}=\O(1)\quad\mbox{as $t\to-\infty$},
\end{align*}
which mean condition~(B2) with $\lambda_-=\tilde{\lambda}_1$
 and $\lambda_+=\tilde{\lambda}_4$.
Thus, we complete the proof.
\end{proof}

\section{Analytic Conditions}
{\color{black}In this section we consider the general case of $n\ge 2$
 and extend the Melnikov method of Part I to reversible systems under assumptions~(R1)-(R4).
Here we restrict to $\Rset$ the domain on which the VE \eqref{eqn:ve} is defined.}

\subsection{Extension of Melnikov's method}

Consider the general case of $n\ge 2$ and assume (R1)-(R4).
By assumption~(R1) there exists a splitting $\Rset^{2n}=\fix(R)\oplus\fix(-R)$.
So we choose a scalar product $\langle\cdot,\cdot\rangle$ in $\Rset^{2n}$ such that
\[
\fix(-R)=\fix(R)^\bot.
\]
{\color{black}Since $f(Rx;0)+Rf(x;0)=0$, we have
\begin{equation}
\D_xf(x^\h(t);0)R+R\D_xf(x^\h(t);0)=0.
\label{eqn:2a}
\end{equation}\color{black}
It follows from \eqref{eqn:2a} that
 if $\xi(t)$ is a solution to \eqref{eqn:ve},
 then so are $\pm R\xi(-t)$ as well as $-\xi(t)$.
For \eqref{eqn:ve},
 we also say that a solution $\xi(t)$ is \emph{symmetric} and \emph{antisymmetric}
 if $\xi(t)=R\xi(-t)$ and $\xi(t)=-R\xi(-t)$, respectively,
 and show that it is symmetric and antisymmetric
 if and only if it intersects the spaces $\fix(R)$ and $\fix(-R)=\fix(R)^\bot$, respectively,
 at $t=0$. 
We easily see that $\xi=\dot{x}^\h(t)$ is antisymmetric since $x^\h(t)=Rx^\h(-t)$ so that
\[
\dot{x}^\h(t)=-R\dot{x}^\h(-t).
\]
Here we also assume the following.}


\begin{enumerate}
\setlength{\leftskip}{-0.5em}
\item[\bf(R5)]
Let $n_0<2n$ be a positive integer.
The VE \eqref{eqn:ve} has just $n_0$ linearly independent bounded solutions,
 $\xi=\varphi_1(t)\,(=\dot{x}^\h(t)),\varphi_2(t),\ldots,\varphi_{n_0}(t)$,
 such that $\varphi_j(0)\in\fix(R)$ 
 for $j=2,\ldots,{\color{black}n_0}$.
If $n_0=1$, then there is no bounded solution
 that is linearly independent of $\xi=\dot{x}^\h(t)$.
\end{enumerate}

Here by abuse of notation $\varphi_j(t)$, $j=1,\ldots,2n$,
 are different from those of Lemma~2.1
 (such abuse of notation was used in Part~I without mentioning).
Note that $\varphi_1(0)=\dot{x}^\h(0)\in\fix(-R)=\fix(R)^\bot$.
Thus, $\varphi_2(t),\ldots,\varphi_{n_0}(t)$ are symmetric but $\varphi_1(t)$ is antisymmetric.
Using Lemma~2.1 of Part~I, under assumptions~(R1)-(R5),
 we can take other linearly independent solutions $\varphi_j(t)$, $j=n_0+1,\ldots,n$,
 to the VE \eqref{eqn:ve} than those given in (R5) as follows.

\begin{lem}
\label{lem:3a}
{\color{black}
There exist linearly independent solutions $\varphi_j(t)$, $j=1,\ldots,2n$, to \eqref{eqn:ve}
 such that they satisfy} the following conditions:
\begin{equation}
\begin{array}{ll}
\displaystyle
\lim_{t\to+\infty}|\varphi_j(t)|=0,\quad
\lim_{t\to-\infty}|\varphi_j(t)|=\infty
&
\mbox{for $j=n_0+1,\ldots,n$};\\
\displaystyle
\lim_{t\to\pm\infty}|\varphi_j(t)|=\infty,\quad
\varphi_j(0)\in\fix(R)
&
\mbox{for $j=n+1$};\\
\displaystyle
\lim_{t\to\pm\infty}|\varphi_j(t)|=\infty,\quad
\varphi_j(0)\in\fix(-R)
&
\mbox{for $j=n+2,\ldots,n+n_0$};\\\displaystyle
\lim_{t\to+\infty}|\varphi_j(t)|=\infty,\quad
\lim_{t\to-\infty}|\varphi_j(t)|=0
&
\mbox{for $j=n+n_0+1,\ldots,2n$}.
\end{array}
\label{eqn:vphi1}
\end{equation}
{\color{black}Here $\varphi_j(t)$, $j=1,\ldots,n_0$, are given in (R5).}
\end{lem}

\begin{proof}
It follows from Lemma~2.1 of Part~I that
 there are linearly independent solutions $\varphi_j(t)$, $j=n_0+1,\ldots,n+n_0$, to \eqref{eqn:ve}
 such that they are linearly independent of $\varphi_j(t)$, $j=1,\ldots,n_0$,
 and satisfy the first, second and third conditions in \eqref{eqn:vphi1}
 except that $\varphi_j(0)\in\fix(R)$ or $\fix(-R)$ for $j=n_0+1,\ldots,n+n_0$.
Note that other linearly independent solutions with $\xi(0)\in\fix(R)$
 than $\varphi_j(t)$, $j=2,\ldots,n_0$, do not converge to $0$ as $t\to+\infty$ or $-\infty$.
Let
\begin{equation}
\varphi_{n+j}(t)=R\varphi_j(-t),\quad
j=n_0+1,\ldots,n.
\label{eqn:vphi2}
\end{equation}
We easily see that they satisfy the fourth condition in \eqref{eqn:vphi1}
 and $\varphi_j(t)$, $j=1,\ldots,2n$ are linearly independent.

Let $\xi=\varphi(t)$ be a solution to \eqref{eqn:ve}.
If $\varphi(0)\not\in\fix(-R)$ and $\varphi(0)\not\in\fix(R)$,
 then $\xi(t)=\varphi(t)+R\varphi(-t)$ and $\xi(t)=\varphi(t)-R\varphi(-t)$
 satisfy $\xi(0)\in\fix(R)$ and $\xi(0)\in\fix(-R)$, respectively.
Hence, we choose $\varphi_j(t)$, $j=n+1,\ldots,n+n_0$,
 such that $\varphi_j(0)\in\fix(R)\cup\fix(-R)$.
Moreover, the subspace spanned by $\varphi_j(0)$ and $\varphi_{n+j}(0)$, $j=n_0+1,\ldots,n$,
 intersects each of $\fix(R)$ and $\fix(-R)$ in $(n-n_0)$-dimensional subspaces.
Thus, one of $\varphi_j(t)$, $j=n+1,\ldots,n+n_0$, is contained in $\fix(R)$,
 and the others are contained in $\fix(-R)$
 since $\varphi_1(0)\in\fix(-R)$, $\varphi_j(0)\in\fix(R)$, $j=2,\ldots,n_0$,
 and $\dim\fix(R)=\dim\fix(-R)=n$.
This completes the proof.
\end{proof}

Let $\Phi(t)=(\varphi_1(t),\ldots,\varphi_{2n}(t))$.
Then $\Phi(t)$ is a fundamental matrix to \eqref{eqn:ve}.
Define $\psi_j(t)$, $j=1,\ldots,2n$, by
\begin{equation}
\langle\psi_j(t),\varphi_k(t)\rangle=\delta_{jk},\quad
j,k=1,\ldots,2n,
\label{eqn:psi}
\end{equation}
where $\delta_{jk}$ is Kronecker's delta.
The functions $\psi_j(t)$, $j=1,\ldots,n$, can be obtained
 by the formula $\Psi(t)=(\Phi^\ast(t))^{-1}$,
 where $\Psi(t)=(\psi_1(t),\ldots,\psi_n(t))$
 and $\Phi^\ast(t)$ is the transpose matrix of $\Phi(t)$.
It immediately follows from (R5) and  \eqref{eqn:vphi1}-\eqref{eqn:psi} that
\[
\begin{array}{ll}
\displaystyle
 \lim_{t\rightarrow\pm\infty}|\psi_j(t)|=\infty,\quad
\psi_j(0)\in\fix(-R^\ast) & \mbox{for $j=1$};\\
\displaystyle
 \lim_{t\rightarrow\pm\infty}|\psi_j(t)|=\infty,\quad
\psi_j(0)\in\fix(R^\ast) & \mbox{for $j=2,\ldots,n_0$};\\
\displaystyle
 \lim_{t\rightarrow+\infty}|\psi_j(t)|=\infty,\quad
 \lim_{t\rightarrow-\infty}|\psi_j(t)|=0
&
\mbox{for $j=n_0+1,\ldots,n$};\\
\displaystyle
 \lim_{t\rightarrow\pm\infty}|\psi_j(t)|=0,\quad
\psi_j(0)\in\fix(R^\ast)
&
\mbox{for $j=n+1$};\\
\displaystyle
 \lim_{t\rightarrow\pm\infty}|\psi_j(t)|=0,\quad
\psi_j(0)\in\fix(-R^\ast)
&
\mbox{for $j=n+2,\ldots,n+n_0$};\\
\displaystyle
 \lim_{t\rightarrow \infty}|\psi_j(t)|=0,\quad
 \lim_{t\rightarrow -\infty}|\psi_j(t)|=\infty
&
\mbox{for $j=n+n_0+1,\ldots,2n$}
\end{array}
\]
and
\begin{equation}
\psi_{n+j}(t)=R^\ast\psi_j(-t),\quad
j=n_0+1,\ldots,n.
\label{eqn:psi2}
\end{equation}
Moreover, $\Psi(t)$ is a fundamental matrix to the adjoint equation
\begin{equation}
\dot{\xi}=-\D_x f(x^\h(t);0)^\ast \xi.
\label{eqn:ave}
\end{equation}
See Section~2.1 of Part I.
Note that if $\xi(t)$ is a solution to \eqref{eqn:ave},
 then so are $\pm R^\ast\xi(-t)$ as well as $-\xi(t)$.

As in Part~I,
 we look for a symmetric homoclinic orbit of the form
\begin{equation}
x=x^\h(t)+\sum_{j=1}^{n_0-1}\alpha_j\varphi_{j+1}(t)+\O(\sqrt{|\alpha|^4+|\mu|^2})
\label{eqn:hoa}
\end{equation}
satisfying $x(0)\in\fix(R)$ in (\ref{eqn:sys}) when $\mu\neq 0$,
 where $\alpha=(\alpha_1,\ldots,\alpha_{n_0-1})$.
Here the $\O(\alpha)$-terms are eliminated in \eqref{eqn:hoa} if $n_0=1$.
Let $\kappa$ be a positive real number such that $\kappa<\frac{1}{4}\lambda_1$,
and define two Banach spaces as
\[
\begin{split}
\hat{\Z}^0=\{z\in C^0(\Rset,\Rset^n)\mid&
 \sup_{t\ge 0}|z(t)|e^{\kappa|t|}<\infty,z(t)=-Rz(-t)\},\\
\hat{\Z}^1=\{z\in C^1(\Rset,\Rset^n)\mid&
\sup_{t\ge 0}|z(t)|e^{\kappa|t|},
 \sup_{t\ge 0}|\dot{z}(t)|e^{\kappa|t|}<\infty,z(t)=Rz(-t)\},
\end{split}
\]
where the maximum of the suprema is taken as a norm of each space.
We have the following result as in Lemma~2.3 of Part~I.

\begin{lem}
\label{lem:4a}
The nonhomogeneous VE,
\begin{equation}
\dot{\xi}=\D_x f(x^\h(t);0)\xi+\eta(t)
\label{eqn:nhve}
\end{equation}
with $\eta\in\hat{\Z}^0$, has a solution in $\hat{\Z}^1$ if and only if
\begin{equation}
\int_{-\infty}^{\infty}
 \langle\psi_{n+j}(t),\eta(t)\rangle\,dt=0,\quad
j=2,\ldots,n_0.
\label{eqn:lem2b}
\end{equation}
Moreover, if condition \eqref{eqn:lem2b} holds,
 then there exists a unique solution to \eqref{eqn:nhve}
 satisfying $\langle\psi_j(0),\xi(0)\rangle=0$, $j=1,\ldots,n_0$, in $\hat{\Z}^1$.
\end{lem}

\begin{proof}
As in Lemma~2.2 of Part I, we see that if $z\in\hat{\Z}^1$, then
\begin{equation}
\int_{-\infty}^\infty\langle\psi_{n+j}(t),\dot{z}(t)-\D_xf(x^\h(t);0)z(t)\rangle\d t=0,\quad
j=2,\ldots,n_0.
\label{eqn:lem2c}
\end{equation}
Hence, if Eq.~\eqref{eqn:nhve} has a solution $\xi\in\hat{\Z}^1$, then
\[
\int_{-\infty}^\infty\langle\psi_{n+j}(t),\eta(t)\rangle\d t
=\int_{-\infty}^\infty\langle\psi_{n+j}(t),\dot{\xi}(t)-\D_xf(x^\h(t);0)\xi(t)\rangle\d t=0
\]
for $j=2,\ldots,n_0$.
Thus, the necessity of the first part is proven.

Assume that condition~\eqref{eqn:lem2b} holds.
We easily see that for $\eta\in\hat{\Z}^0$
\begin{align*}
\int_{-\infty}^0\langle\psi_{n+j}(t),\eta(t)\rangle\d t
=& \int_{-\infty}^0\langle-R^\ast\psi_{n+j}(-t),-R\eta(-t)\rangle\d t\\
=& \int_0^{\infty}\langle \psi_{n+j}(t),\eta(t)\rangle\d t=0,\quad
j=2,\ldots,n_0,
\end{align*}
while
\begin{align*}
\int_{-\infty}^0\langle\psi_{n+1}(t),\eta(t)\rangle\d t
=& \int_{-\infty}^0\langle R^\ast\psi_{n+1}(-t),-R\eta(-t)\rangle\d t\\
=& -\int_0^{\infty}\langle \psi_{n+1}(t),\eta(t)\rangle\d t.
\end{align*}
Hence,
\begin{align*}
\hat{\xi}(t)=& \left(\sum_{j=1}^{n_0}+\sum_{j=n+2}^{n+n_0}\right)
 \varphi_j(t)\int_0^t\langle\psi_j(s),\eta(s)\rangle\d s\\
& +\sum_{j=n_0+1}^{n+1}\varphi_j(t)\int_{-\infty}^t\langle\psi_j(s),\eta(s)\rangle\d s
 -\sum_{j=n+n_0+1}^{2n}\varphi_j(t)\int_t^\infty\langle\psi_j(s),\eta(s)\rangle\d s
\end{align*}
is a solution to \eqref{eqn:nhve} and contained in $\hat{\Z}^1$
 since by \eqref{eqn:vphi2} and \eqref{eqn:psi2}
\begin{align*}
\hat{\xi}(0)=
\sum_{j=n_0+1}^{n+1}&\varphi_j(0)\int_{-\infty}^0\langle\psi_j(s),\eta(s)\rangle\d s\\
& -\sum_{j=n+n_0+1}^{2n}\varphi_j(0)\int_0^\infty\langle\psi_j(s),\eta(s)\rangle\d s\\
=\sum_{j=n_0+1}^{n+1}&\varphi_j(0)\int_{-\infty}^0\langle\psi_j(s),\eta(s)\rangle\d s\\
& +\sum_{j=n_0+1}^{n}R\varphi_j(0)\int_0^\infty\langle R^\ast\psi_j(-s),R\eta(-s)\rangle\d s\\
=\sum_{j=n_0+1}^{n}&(\varphi_j(0)+R\varphi_j(0))\int_{-\infty}^0\langle\psi_j(s),\eta(s)\rangle\d s\\
& +\varphi_{n+1}(0)\int_{-\infty}^0\langle\psi_{n+1}(s),\eta(s)\rangle\d s
\in\fix(R).
\end{align*}
The sufficiency of the first part is thus proven.

We turn to the second part.
Obviously, $\langle\psi_j(0),\hat{\xi}(0)\rangle=0$, $j=1,\ldots,n_0$.
In addition, if $\xi=\xi(t)$ is a solution to \eqref{eqn:nhve}, then so is $\xi=R\xi(-t)$,
 and if $\xi(t)\in\hat{\Z}^1$, then $\xi(0)\in\fix(R)$ so that $\langle\psi_1(0),\xi(0)\rangle=0$
 by $\fix(-R)^\bot=\fix(R)$.
Moreover, any solution to \eqref{eqn:nhve} is represented
 as $\xi(t)=\hat{\xi}(t)+\sum_{j=1}^{2n}d_j\varphi_j(t)$,
 where $d_j\in\Rset$, $j=1,\ldots,n$, are constants,
 but one has $d_j=0$, $j=1,\ldots,2n$, if it is contained in $\hat{\Z}^1$
 and satisfies $\langle\psi_j(0),\xi(0))\rangle=0$, $j=2,\ldots,n_0$.
This completes the proof.
\end{proof}

Let
\[
\hat{\Z}_0^1=\{z\in\hat{\Z}^1\mid\langle\psi_j(0),z(0)\rangle=0,j=1,\ldots,n_0\}\subset\hat{\Z}^1,
\]
which is also a Banach space. 
Define a differentiable function
 $F:\hat{\Z}_0^1\times\Rset^{n_0-1}\times\Rset\to\hat{\Z}^0$ as
\begin{align}
F(z;\alpha,\mu)(t)
=&\frac{\d}{\d t}\left(x^\h(t)+z(t)+\sum_{j=1}^{n_0-1}\alpha_j\varphi_{j+1}(t)\right)\notag\\
& -f\left(x^\h(t)+z(t)+\sum_{j=1}^{n_0-1}\alpha_j\varphi_{j+1}(t);\mu\right).
\label{eqn:F}
\end{align}
Note that for $z\in\hat{\Z}_0^1$
\begin{align*}
RF(z;\alpha,\mu)(-t)
=&\left(R\dot{x}^\h(-t)+R\dot{z}(-t)+\sum_{j=1}^{n_0-1}\alpha_jR\dot{\varphi}_{j+1}(-t)\right)\\
&-Rf\left(x^\h(-t)+z(-t)+\sum_{j=1}^{n_0-1}\alpha_j\varphi_{j+1}(-t);\mu\right)\\
=&-\left(\dot{x}^\h(t)+\dot{z}(t)+\sum_{j=1}^{n_0-1}\alpha_j\dot{\varphi}_{j+1}(t)\right)\\
& +f\left(x^\h(t)+z(t)+\sum_{j=1}^{n_0-1}\alpha_j\varphi_{j+1}(t);\mu\right)
=-F(z;\alpha,\mu)(t).
\end{align*}
A solution $z\in\hat{\Z}_0^1$ to
\[
F(z;\alpha,\mu)=0
\]
for $(\alpha,\mu)$ fixed gives a symmetric homoclinic orbit to $x=0$.

We now proceed as in Section~2.1 of Part~I
 with taking the reversibility of \eqref{eqn:sys} into account.
Define a projection $\Pi:\hat{\Z}^0\rightarrow\hat{\Z}^1$ by
\[
\Pi z(t)
=q(t)\sum_{j=2}^{n_0}\left(\int_{-\infty}^\infty
 \langle\psi_{n+j}(\tau),z(\tau)\rangle\,\d\tau\right)\varphi_{n+j}(t),
\]
where $q:\Rset\rightarrow\Rset$ is a continuous function satisfying
\begin{equation}
\sup_t|q(t)|e^{\kappa|t|}<\infty,\quad
q(t)=q(-t)\quad\mbox{and}\quad
\int_{-\infty}^\infty q(t)\d t=1.
\label{eqn:q}
\end{equation}
Note that for $z\in\hat{\Z}^0$
\[
\int_{-\infty}^\infty \langle\psi_{n+1}(\tau),z(\tau)\rangle\,\d\tau=0.
\]
Using Lemma~\ref{lem:4a} and the implicit function theorem,
 we can show that there are a neighborhood $U$ of $(\alpha,\mu)=(0,0)$
 and a differentiable function $\bar{z}:U\rightarrow\hat{\Z}_0^1$
 such that $\bar{z}(0,0)=0$ and 
\begin{equation}
(\id-\Pi)F(\bar{z}(\alpha,\mu);\alpha,\mu)=0
\label{eqn:F1a}
\end{equation}
for $(\alpha,\mu)\in U$, where ``$\id$'' represents the identity.

Let
\begin{equation}
\bar{F}_j(\alpha,\mu)=\int_{-\infty}^\infty
 \langle\psi_{n+j+1}(t),F(\bar{z}(\alpha,\mu);\alpha,\mu)(t)\rangle\,\d t,\quad
j=1,\ldots,n_0-1.
\label{eqn:bF1}
\end{equation}
We can prove the following theorem as in Theorem 2.4 of Part~I
 (see also Theorem~5 of \cite{G92}).

\begin{thm}
\label{thm:ma}
Under assumptions~{\rm(R1)-(R5)} with $n_0\ge 1$,
 suppose that $\bar{F}(0;0)=0$.
Then for each $(\alpha,\mu)$ sufficiently close to $(0,0)$
 Eq.~\eqref{eqn:sys} admits a unique symmetric homoclinic orbit to the origin
  of the form \eqref{eqn:hoa}
\end{thm}

Henceforth we set $m=1$ and apply Theorem~\ref{thm:ma}
 to obtain persistence and bifurcation theorems for symmetric homoclinic orbits
 in \eqref{eqn:sys} with $n\ge 2$, as in Sections~2.2 and 2.3 of Part~I.

\subsection{Persistence and bifurcations of symmetric homoclinic orbits}

We first assume that $n_0=1$,
 which means that condition~(C) does not hold.
Since $\Pi z=0\in\hat{\Z}^0$ for $z\in\hat{\Z}^0$
 and Eq.~\eqref{eqn:F1a} has a solution $\bar{z}(\mu)$ on a neighborhood $U$ of $\mu=0$,
 we immediately obtain the following result from the above argument,
 as in Theorem~2.5 of Part~I.
 
\begin{thm}
\label{thm:ma1}
Under assumptions~{\rm(R1)-(R5)} with $n_0=1$,
 there exists a symmetric homoclinic orbit on some open interval $I\subset\Rset$ including $\mu=0$.
\end{thm}

\begin{rmk}
\label{rmk:ma1}
Theorem~$\ref{thm:ma1}$ implies that
 if condition~{\rm(C)} does not hold,
 then the homoclinic orbit $x^\h(t)$ persists, i.e., no bifurcation occurs,
 as stated in Section~$1$.  
\end{rmk}

We now assume that $n_0=2$, which means that condition~(C) holds
 and no further linearly independent solution to the VE \eqref{eqn:ve} exists.
Define two constants $a_2,b_2$ as
\begin{equation}
\begin{split}
&
a_2=\int_{-\infty}^{\infty}
 \langle\psi_{n+2}(t),\D_\mu f(x^\h(t);0)\rangle\,\d t,\\
&
b_2=\frac{1}{2}\int_{-\infty}^{\infty}
 \langle\psi_{n+2}(t),\D_x^2 f(x^\h(t);0)(\varphi_2(t),\varphi_2(t))\rangle\,\d t
\end{split}
\label{eqn:ab}
\end{equation}
(cf. Eq.~(19) of Part~I).
We obtain the following result as in Theorem~2.7 of Part~I.
 
\begin{figure}[t]
\begin{center}
\includegraphics[scale=0.7]{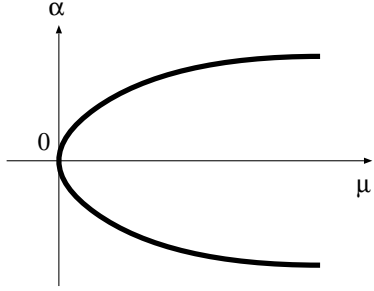}
\caption{Saddle-node bifurcation: Supercritical case is plotted.}
\label{fig:ma2}
\end{center}
\end{figure}

\begin{thm}
\label{thm:ma2}
Under assumptions~{\rm(R1)-(R5)} with $n_0=2$,
 suppose that $a_2,b_2\neq 0$.
Then for some open interval $I$ including $\mu=0$
 there exists a differentiable function $\phi:I\rightarrow\Rset$
 with $\phi(0)=0$, $\phi'(0)=0$ and $\phi''(0)\neq 0$,
 such that a symmetric homoclinic orbit of the form \eqref{eqn:hoa} exists for $\mu=\phi(\alpha)$,
 i.e., a saddle-node bifurcation of symmetric homoclinic orbits occurs at $\mu=0$.
Moreover, it is supercritical and subcritical if $a_2b_2<0$ and $>0$, respectively.
See Fig.~$\ref{fig:ma2}$.
\end{thm}

We next assume the following instead of (R4).
\begin{enumerate}
\setlength{\leftskip}{-0.2em}
\item[\bf(R4')]
The equilibrium $x=0$ has a symmetric homoclinic orbit $x^\h(t;\mu)$
 in an open interval $I\ni \mu=0$.
Moreover, $\langle\psi_{n+2}(t),\dot{x}^\h(t;\mu)\rangle=0$ for any $t\in\Rset$ and $\mu\in I$.
\end{enumerate}
Under assumption~(R4') we have
\begin{align*}
\D_\mu \langle\psi_{n+2}(t),\dot{x}^\h(t;\mu)\rangle\bigg|_{\mu=0}
=&\langle\psi_{n+2}(t),\D_\mu\dot{x}^\h(t;0)\rangle
=\langle\psi_{n+2}(t),\D_\mu f(x^\h(t;0);0)\rangle,
\end{align*}
so that
\begin{equation}
a_2=\int_{-\infty}^{\infty}
 \langle\psi_{n+2}(t),\D_\mu f(x^\h(t;0);0)\rangle\,\d t=0.
\label{eqn:a20}
\end{equation}
In this situation we cannot apply Theorem~\ref{thm:ma2}.
Let $\xi=\xi^\mu(t)$ be {\color{black}the} unique solution to
\begin{equation}
\dot{\xi}=\D_x f(x^\h(t);0)\xi
 +(\id-\Pi)\D_\mu f(x^\h(t);0)
\label{eqn:ximu1}
\end{equation}
in $\hat{\Z}_0^1$, and define
\begin{equation}
\bar{a}_2=
\int_{-\infty}^{\infty}
 \langle\psi_{n+2}(t),
 \D_\mu\D_x f(x^\h(t);0)\varphi_2(t)
 +\D_x^2 f(x^\h(t);0)(\xi^\mu(t),\varphi_2(t))\rangle\,\d t,
\label{eqn:bara}
\end{equation}
where $x^\h(t)=x^\h(t;0)$ (cf. Eq.~(20) of Part~I).

\begin{thm}
\label{thm:ma3}
Under assumptions~{\rm(R1)}-{\rm(R3)}, {\rm(R4')} and {\rm(R5)} with $n_0=2$,
 suppose that $\bar{a}_2,b_2\neq 0$.
Then for some open interval $I$ including $\mu=0$
 there exists a differentiable function $\phi:I\rightarrow\Rset$
 with $\phi(0)=0$ and $\phi'(0)\neq 0$,
 such that a different symmetric homoclinic orbit of the form \eqref{eqn:hoa}
 than $x^\h(t;\mu)$ exists for $\alpha=\phi(\mu)$ {\color{black}with $\mu\neq 0$},
 i.e., a transcritical bifurcation of symmetric homoclinic orbits occurs at $\mu=0$.
See Fig.~$\ref{fig:ma3}$.
\end{thm}

\begin{figure}[t]
\begin{center}
\includegraphics[scale=0.7]{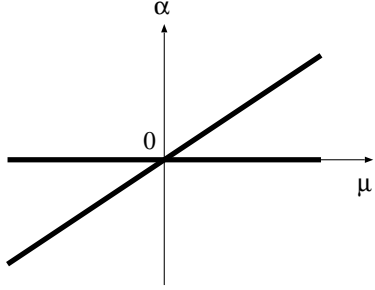}
\caption{Transcritical bifurcation.}
\label{fig:ma3}
\end{center}
\end{figure}

\begin{proof}
Differentiating \eqref{eqn:F1a} with respect to $\alpha$
 and using \eqref{eqn:F}, we have
\[
\D_\alpha(\id-\Pi)F(\bar{z};0,0)
 =\frac{\d}{\d t}\D_\alpha\bar{z}-\D_x f(x^\h(t);0)\D_\alpha\bar{z}=0
\]
at $(\alpha,\mu)=(0,0)$, i.e.,
 $\D_\alpha\bar{z}(0;0)(t)$
 is a solution of \eqref{eqn:ve},
 so that $\D_\alpha\bar{z}(0;0)(t)=0$ by Lemma~\ref{lem:4a}.  
Using this fact, \eqref{eqn:lem2c} and \eqref{eqn:a20}, we compute \eqref{eqn:bF1} as
\begin{align*}
\bar{F}_1(\alpha,\mu)
=&\int_{-\infty}^\infty\langle\psi_{n+2}(t),
 -\D_\mu\D_x f(x^\h(t);0)\varphi_2(t)\mu\\
& \qquad
 -\half\alpha^2\D_x^2 f(x^\h(t);0)(\varphi_2(t),\varphi_2(t))\rangle\d t
 +\O(\sqrt{\alpha^6+|\mu|^4})\\
=&-\bar{a}_2\alpha\mu-b_j\alpha^2+\O(\sqrt{\alpha^6+|\mu|^4}),
\end{align*}
as in the proof of Theorem~2.7 of Part~I.
Since $\bar{F}_1(0,0)=0$ and $\D_\mu\bar{F}_1(0,0)\neq 0$,
 we apply the implicit function theorem to show that
 there exist an open interval $I\,(\ni 0)$
 and a differentiable function $\bar{\phi}:I\rightarrow\Rset$
 such that $\bar{F}(\bar{\phi}(\alpha),\alpha)=0$ for $\alpha\in I$
 with $\bar{\phi}(0)=0$ and $\bar{\phi}'(0)\neq 0$.
This implies the result along with Theorem~\ref{thm:ma}.
\end{proof}

\begin{rmk}
For the class of systems discussed in Part~I, including Hamiltonian systems,
 we can prove a result similar to Theorem~$\ref{thm:ma3}$.
\end{rmk}

Finally we consider the $\Zset_2$-equivalent or equivariant case for $n_0=2$,
 and assume the following.
\begin{enumerate}
\item[\bf(R6)]
Eq.~\eqref{eqn:sys} is {\em $\Zset_2$-equivalent} or {\em equivariant}, i.e.,
 there exists an $n\times n$ matrix $S$
 such that $S^2=\id_n$ and $Sf(x;\mu)=f(Sx;\mu)$.
\end{enumerate}
See Section~2.3 of Part~I or Section~7.4 of \cite{K04}
 for more details on $\Zset_2$-equivalent or equivariant systems.
Especially, if $x=\bar{x}(t)$ is a solution to \eqref{eqn:sys}, then so is $x=S\bar{x}(t)$.
We say that the pair $\bar{x}(t)$ and $S\bar{x}(t)$ are $S$-conjugate if $\bar{x}(t)\neq S\bar{x}(t)$.
The space $\Rset^{\color{black}2n}$ can be decomposed into a direct sum as
\[
\Rset^{\color{black}2n}=X^+\oplus X^-,
\]
where $Sx=x$ for $x\in X^+$ and $Sx=-x$ for $x\in X^-$.
We also need the following assumption.
\begin{enumerate}
\item[\bf(R7)]
We have $X^-=(X^+)^\bot$.
For every $t\in\Rset$,
 $x^\h(t),\psi_{n+1}(t)\in X^+$
 and $\varphi_2(t),\psi_{n+2}(t)\in X^-$.
\end{enumerate}
In Part~I, we implicitly assumed that $X^-=(X^+)^\bot$.
Recall that the scalar product in $\Rset^{2n}$
 was already chosen such that $\fix(-R)=\fix(R)^\bot$.

Assumption~(R7) also means that $\varphi_1(t)\in X^+$.
Moreover, a symmetric homoclinic orbit of the form (\ref{eqn:hoa}) has an $S$-conjugate counterpart
 for $\alpha\neq 0$ since it is not included in $X^+$.
In  this situation, we have $a_2,b_2=0$ in Theorems~\ref{thm:ma2} and \ref{thm:ma3},
 as in Lemma~2.8 of Part~I,  and cannot apply these theorems.

Let $\xi=\xi^\alpha(t)$ be a unique solution to
\begin{equation}
\dot{\xi}=\D_x f(x^\h(t);0)\xi
 +\frac{1}{2}(\id-\Pi)\D_x^2 f(x^\h(t);0)(\varphi_2(t),\varphi_2(t))
\label{eqn:xia}
\end{equation}
in $\hat{\Z}_0^1$, and define
\begin{align}
\bar{b}_2 =& \int_{-\infty}^{\infty}
 \Bigl\langle\psi_{n+2}(t),
 \frac{1}{6}\D_x^3 f(x^\h(t);0)(\varphi_2(t),\varphi_2(t),\varphi_2(t))\notag\\
&\qquad
 +\D_x^2 f(x^\h(t);0)(\xi^\alpha(t),\varphi_2(t))\Bigr\rangle\,\d t
\label{eqn:barb}
\end{align}
(cf. Eq.~(20) of Part~I).
We obtain the following result as in Theorem~2.9 of Part~I.

\begin{figure}[t]
\begin{center}
\includegraphics[scale=0.7]{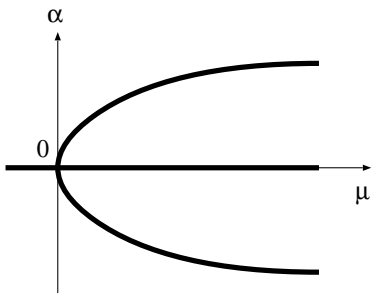}
\caption{Pitchfork bifurcation: Supercritical case is plotted.}
\label{fig:ma4}
\end{center}
\end{figure}

\begin{thm}
\label{thm:ma4}
Under assumptions {\rm(R1)-(R7)} with $n_0=2$,
 suppose that $\bar{a}_2,\bar{b}_2\neq 0$.
Then for $j=1,2$ 
 there exist an open interval $I_j\ni 0$ and a differentiable function
 $\phi_j:I_j\rightarrow\Rset$ with $\phi_j(0)=0$,
 $\phi_2'(0)=0$, $\phi_2''(0)\neq 0$ and $\phi_2(\alpha)=\phi_2(-\alpha)$ for $\alpha\in I_2$,
 such that a symmetric homoclinic orbit exists on $X^+$ for $\mu=\phi_1(\mu_2)$
 and an $S$-conjugate pair of symmetric homoclinic orbits exist for $\mu=\phi_2(\alpha)$:
 a pitchfork bifurcation of homoclinic orbits occurs.
Moreover, it is supercritical and subcritical
 if $\bar{a}_2\bar{b}_2<0$ and $>0$, respectively.
See Fig.~$\ref{fig:ma4}$.
\end{thm}

From Theorems~{\color{black}$\ref{thm:ma2}$, $\ref{thm:ma3}$ and $\ref{thm:ma4}$}
 we see that if condition~{\rm(C)} holds,
 then a saddle-node, transcritical or pitchfork bifurcation occurs
 under some nondegenerate condition, as stated in Section~1.  

\section{Example}

We now illustrate our theory for the four-dimensional system
\begin{equation}
\begin{split}
&
\dot{x}_1=x_3,\quad
\dot{x}_2=x_4,\\
&
\dot{x}_3=x_1-(x_1^2+8x_2^2)x_1-\beta_2x_2,\\
&
\dot{x}_4=sx_2-\beta_1(x_1^2+2x_2^2)x_2-\beta_2x_1-\beta_3x_2^2,
\end{split}
\label{eqn:ex}
\end{equation}
where $s>0$ and $\beta_j$, $j=1$-$4$, are constants.
Similar systems were treated in Part~I and \cite{SY20,YW06,YY20}
 (although $s<0$ in \cite{YW06}).
Eq.~\eqref{eqn:ex} is reversible with the involution
\[
R:(x_1,x_2,x_3,x_4)\mapsto(x_1,x_2,-x_3,-x_4),
\]
for which $\fix(R)=\{(x_1,x_2,x_3,x_4)\in\Rset^4\mid x_3,x_4=0\}$,
 and has an equilibrium at the origin $x=0$.
Thus, assumptions~(R1) and (R2) hold.
The Jacobian matrix of the right hand side of \eqref{eqn:ex} at $x=0$
 has two pairs of positive and negative eigenvalues with the same absolute values
 so that the origin $x=0$ is a hyperbolic saddle.
Thus, assumption~(R3) holds.

Suppose that $\beta_2=0$.
The $(x_1,x_3)$-plane is invariant under the flow of \eqref{eqn:ex}
 and there exist a pair of symmetric homoclinic orbits
\[
x_\pm^\h(t)
 =(\pm\sqrt{2}\,\sech t,0,\mp\sqrt{2}\,\sech\,t\tanh t,0)
\]
to $x=0$.
Thus, assumption~(R4) holds as well as conditions~(B2) and (B3) by Proposition~\ref{prop:2a}.
Henceforth we only treat the homoclinic orbit $x_+^\h(t)$ for simplification
 and denote it by  $x^\h(t)$.
Note that a pair of symmetric homoclinic orbits also exist on the $(x_2,x_4)$-plane.
The VE \eqref{eqn:ve} around $x=x^\h(t)$ for \eqref{eqn:ex} is given by
{\setcounter{enumi}{\value{equation}}
\addtocounter{enumi}{1}
\setcounter{equation}{0}
\renewcommand{\theequation}{\arabic{section}.\theenumi\alph{equation}}
\begin{align}
&
\dot{\xi}_1=\xi_3,\quad
\dot{\xi}_3=(1-6\,\sech^2 t)\xi_1,\label{eqn:exve1}\\
&
\dot{\xi}_2=\xi_4,\quad
\dot{\xi}_4=(s-2\beta_1\,\sech^2 t)\xi_2.\label{eqn:exve2a}
\end{align}

\setcounter{equation}{\value{enumi}}}

As discussed in Section~5 of Part~I (see also \cite{SY20}),
 Eq.~\eqref{eqn:exve2a} has a bounded symmetric solution,
 so that assumption~(R5) holds with $n_0=2$, if and only if
\begin{equation}
\beta_1=\frac{(2\sqrt{s}+4\ell+1)^2-1}{8},\quad
\ell\in\Nset\cup\{0\},
\label{eqn:excon}
\end{equation}
while Eq.~\eqref{eqn:exve1} always has a bounded solution
 corresponding to $\xi=\dot{x}^\h(t)$.
The bounded symmetric solution $(\bar{\xi}_2(t),\bar{\xi}_4(t))$ to \eqref{eqn:exve2a} is given by
\[
\bar{\xi}_2(t)=\sech^{\sqrt{s}}t
\]
for $\ell=0$,
\[
\bar{\xi}_2(t)=\sech^{\sqrt{s}}t\left(1-\left(\sqrt{s}+\frac{3}{2}\right)\sech^2 t\right)
\]
for $\ell=1$,
\[
\bar{\xi}_2(t)=\sech^{\sqrt{s}}t\left(1-2(\sqrt{s}+5)\sech^2 t
 +\left(\sqrt{s}+\frac{5}{2}\right)\left(\sqrt{s}+\frac{7}{2}\right)\sech^4 t\right)
\]
for $\ell=2$ and $\bar{\xi}_4(t)=\dot{\bar{\xi}}_2(t)$ (see Appendix~A of Part~I).
Note that Eq.~\eqref{eqn:exve2a} has an asymmetric bounded solution
 if the first equation \eqref{eqn:excon} holds for $\ell\in\frac{1}{2}\Nset\setminus\Nset$.
Moreover, if condition~\eqref{eqn:excon} holds,
 then the differential Galois group of the VE given by \eqref{eqn:exve1} and \eqref{eqn:exve2a}
 is triangularizable.
See Fig.~7 of Part~I for the dependence of $\beta_1$ satisfying \eqref{eqn:excon} on $s$
 (the definition of $\ell$ there is different from here: $\ell$ is replaced with $2\ell$).
When condition~\eqref{eqn:excon} holds,
 we have
\[
\varphi_2(t)=(0,\bar{\xi}_2(t),0,\bar{\xi}_4(t))
\]
and
\[
\psi_4(t)=(0,-\bar{\xi}_4(t),0,\bar{\xi}_2(t)).
\]
 
Fix the values of $\beta_1$ and $\beta_3\neq 0$ such that Eq.~\eqref{eqn:excon} holds.
Take $\mu=\beta_2$ as a control parameter.
Eq.~\eqref{eqn:ab} become
\begin{align*}
a_2=-\int_{-\infty}^\infty\bar{\xi}_2(t)x_1^\h(t)\d t,\quad
b_2=-\beta_3\int_{-\infty}^\infty\bar{\xi}_2(t)^3\d t.
\end{align*}
See Appendix~A of Part~I for analytic expressions of  these integrals for $\ell=0,1,2$,
 which correspond to $\ell=0,2,4$ there.
Applying Theorem~\ref{thm:ma2},
 we see that a saddle-node bifurcation of symmetric homoclinic orbits occurs at $\beta_2=0$
 if $a_2b_2\neq 0$, which holds for almost all values of $s$ when $\beta_3\neq 0$ and $0\le\ell\le 2$.

We next assume that $\beta_2=0$.
Then assumption~(R4') holds.
Take $\mu=\beta_1$ as a control parameter.
Since $\D_\mu f(x^\h(t);0)=0$,
 the solution to \eqref{eqn:ximu1} in $\tilde{\Z}_0^1$ is $\xi^{\beta_1}(t)=0$.
Eq.~\eqref{eqn:bara} becomes
\begin{align*}
\bar{a}_2=-\int_{-\infty}^\infty\bar{\xi}_2(t)^2x_1^\h(t)^2\d t<0.
\end{align*}
Applying Theorem~\ref{thm:ma3},
 we see that a transcritical bifurcation of symmetric homoclinic orbits occurs
 at the values of $\beta_1$ given by \eqref{eqn:excon} if $b_2\neq 0$.

We next assume that $\beta_2,\beta_3=0$.
Then Eq.~\eqref{eqn:ex} is $\Zset_2$-equivariant with the involution
\[
S: (x_1,x_2,x_3,x_4)\mapsto(x_1,-x_2,x_3,-x_4)
\]
and assumptions~(R6) and (R7) hold.
In particular, $X^+=\{x_2,x_4=0\}$ and $X^-=\{x_1,x_3=0\}$.
Since
\[
\D_x^2 f(x^\h(t);0)(\varphi_2(t),\varphi_2(t))
 =(0,0,x_1^\h(t)\bar{\xi}_2(t)^2,\beta_2\bar{\xi}_2(t)^2)^\ast,
\]
 we write \eqref{eqn:barb} as
\begin{equation}
\bar{b}_2=-2\beta_1\int_{-\infty}^\infty x_1^\h(t)\xi_1^\alpha(t)\bar{\xi}_2(t)^2\d t
 -2\beta_1\int_{-\infty}^\infty\bar{\xi}_2(t)^4\d t,
\label{eqn:exb2}
\end{equation}
where $\xi_1^\alpha(t)$ is the first component of the solution to \eqref{eqn:xia} in $\tilde{\Z}_0^1$
 and given by
 \[
\xi_1^\alpha(t)=\varphi_{11}(t)\int_0^t\psi_{13}(\tau)x_1^\h(\tau)\bar{\xi}_2(\tau)^2\d\tau
-\varphi_{31}(t)\int_t^\infty\psi_{33}(\tau)x_1^\h(\tau)\bar{\xi}_2(\tau)^2\d\tau,
\]
and $\varphi_{jk}(t)$ and $\psi_{jk}(t)$ are the $k$th components of $\varphi_j(t)$ and $\psi_j(t)$,
 respectively (the corresponding formula in Part~I had a small error).
We compute \eqref{eqn:exb2} as
\[
\bar{b}_2=\frac{\sqrt\pi\,\Gamma(2\sqrt s)}{\Gamma(2\sqrt s+\frac{1}{2})}
\frac{P_\ell(\sqrt s)}{Q_\ell(\sqrt s)},
\]
where
\begin{align*}
P_0(x)=&x(x^2-x-1),\\
P_1(x)=&145x^6+530x^5+115x^4-1971x^3-3502x^2-2427x-630,\\
P_2(x)=&27(16627 x^9+ 242984 x^8+ 1310501 x^7 + 2451387 x^6-4949646 x^5 \\
& - 15422381 x^4 -76574432x^3-429952220 x^2 - 49776200 x-12012000),\\
&\hspace*{2em}\vdots
\end{align*}
and
\[
Q_0(x)=1,\quad
Q_\ell(x)=\prod_{j=1}^\ell(x+j)^3\prod_{j=1}^{4\ell}(4x+2j-1)\quad\mbox{for $\ell\ge 1$}.
\]
See Section~7 and Appendix~B of \cite{YY20}
 for derivation of these expressions.
In particular, we see that $\bar{b}_2\neq 0$ except for a finite number of values of $s>0$,
 for each $\ell\ge 0$.
Applying Theorem~\ref{thm:ma4},
 we see that a pitchfork bifurcation of symmetric homoclinic orbits occurs
 at the values of $\beta_1$ given by \eqref{eqn:excon} if $\bar{b}_2\neq 0$.

\begin{figure}[t]
\begin{center}
\includegraphics[scale=0.53]{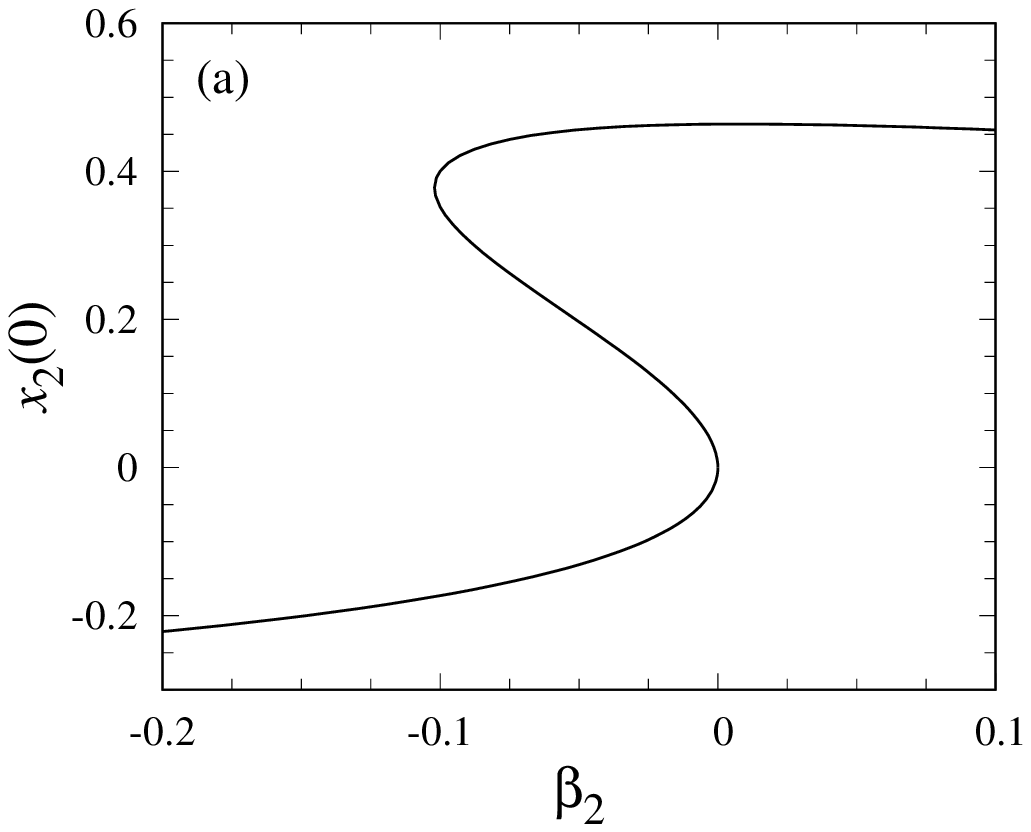}\quad
\includegraphics[scale=0.53]{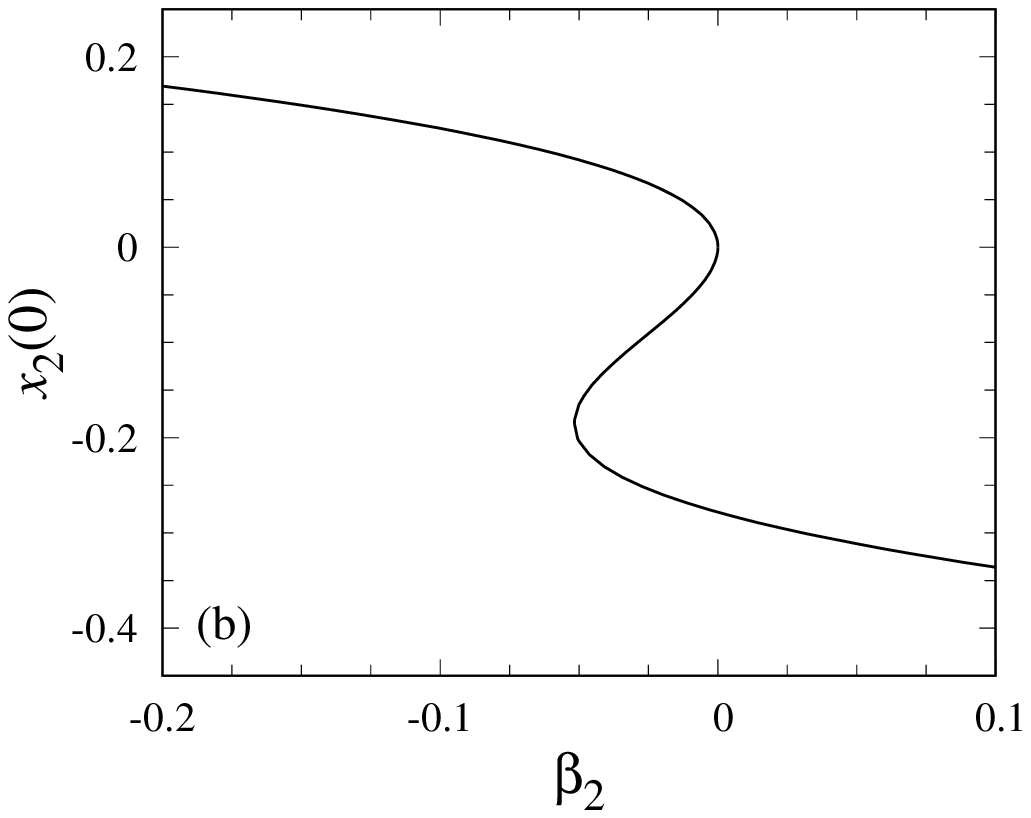}\\
\includegraphics[scale=0.53]{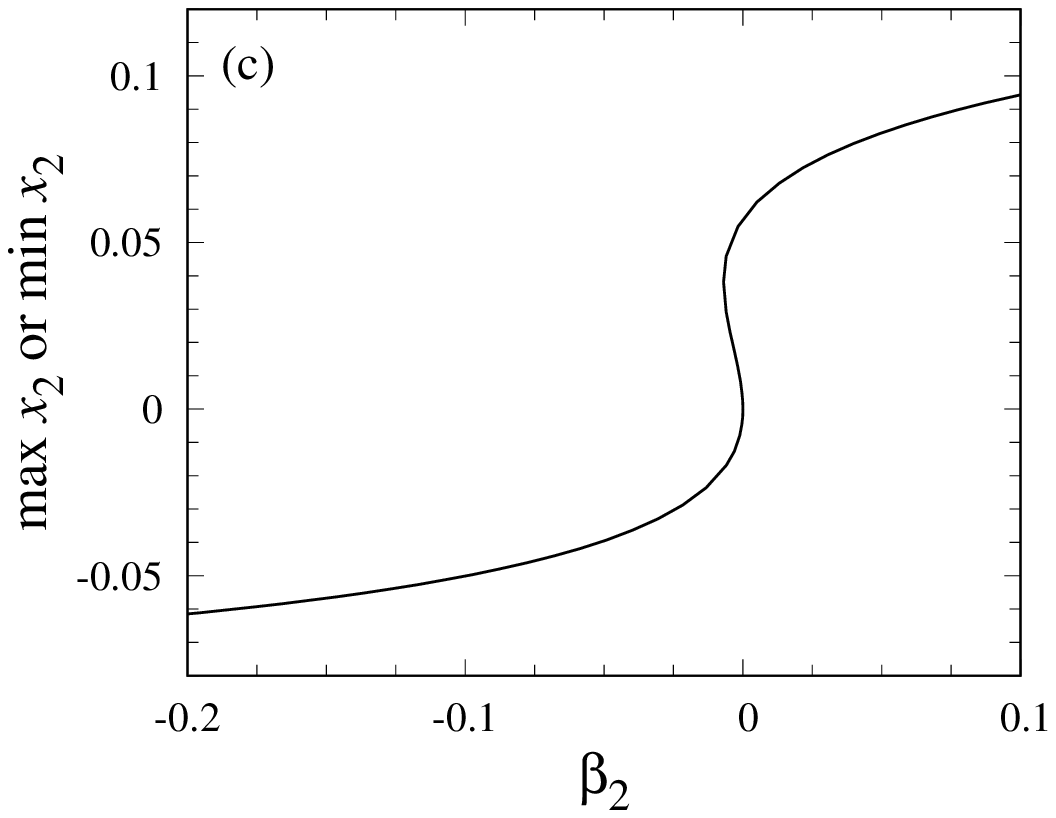}
\caption{Bifurcation diagrams for $s=2$ and $\beta_3=4$:
(a) $\ell=0$; (b) $\ell=1$; (c) $\ell=2$.
Here $\beta_2$ is taken as a control parameter.
\label{fig:4a}}
\end{center}
\end{figure}
\begin{figure}[t]
\begin{center}
\includegraphics[scale=0.48]{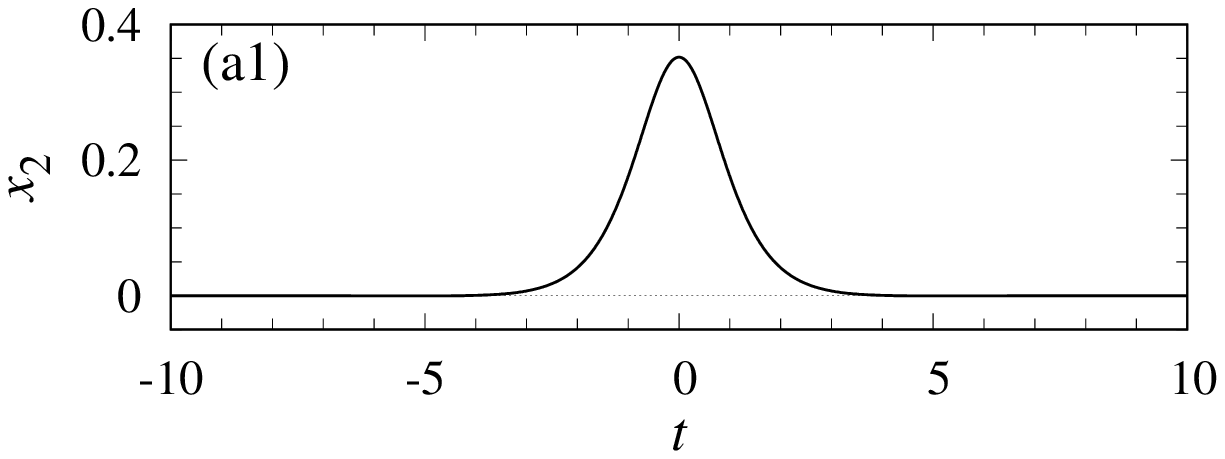}\
\includegraphics[scale=0.48]{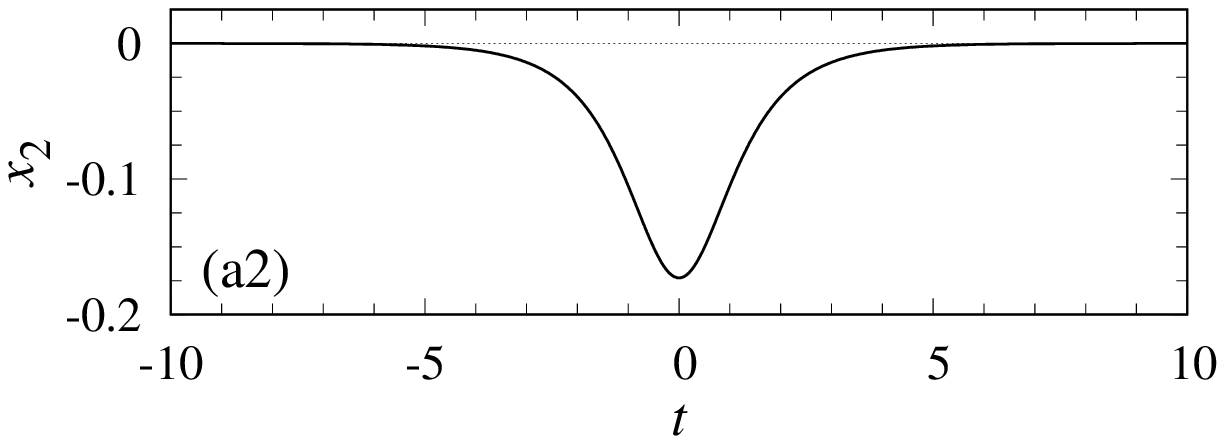}\\[1ex]
\includegraphics[scale=0.48]{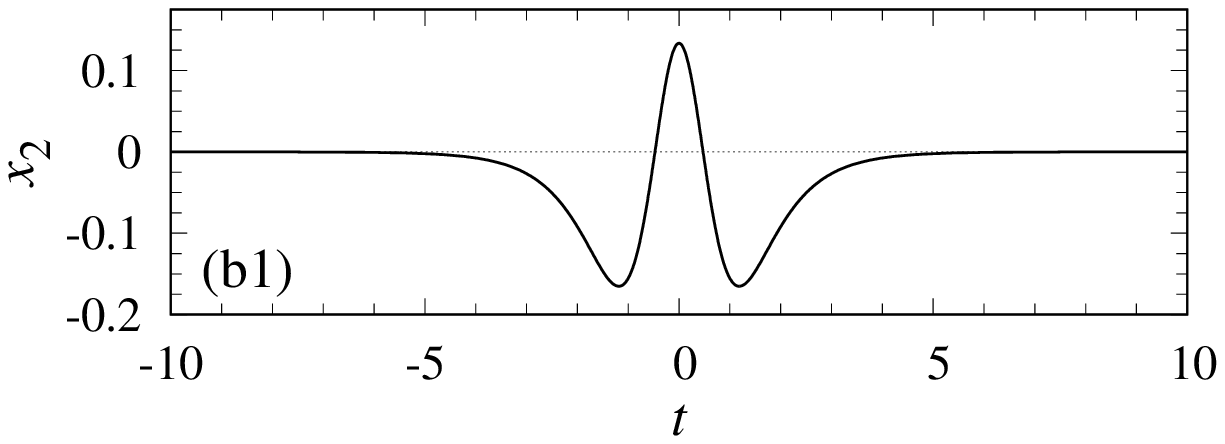}\
\includegraphics[scale=0.48]{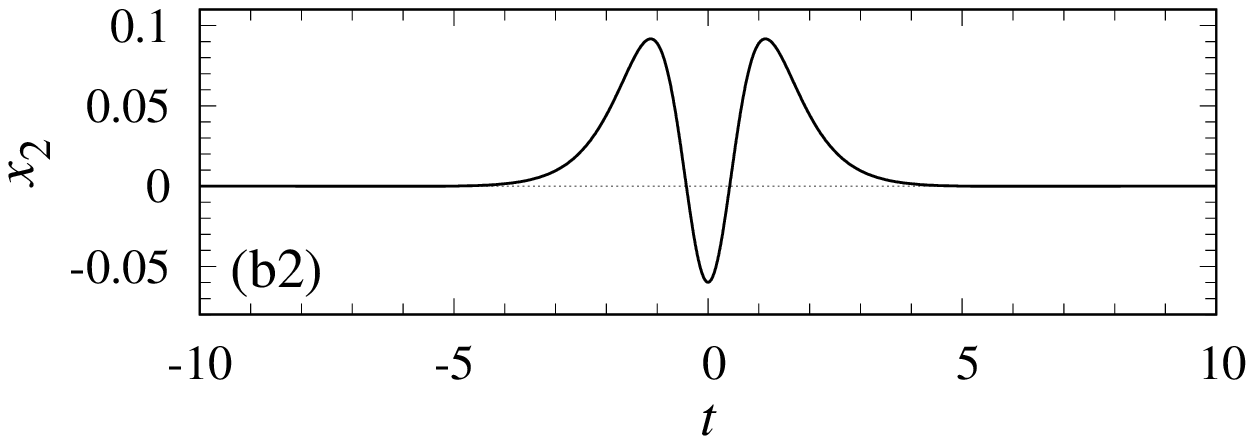}\\[1ex]
\includegraphics[scale=0.48]{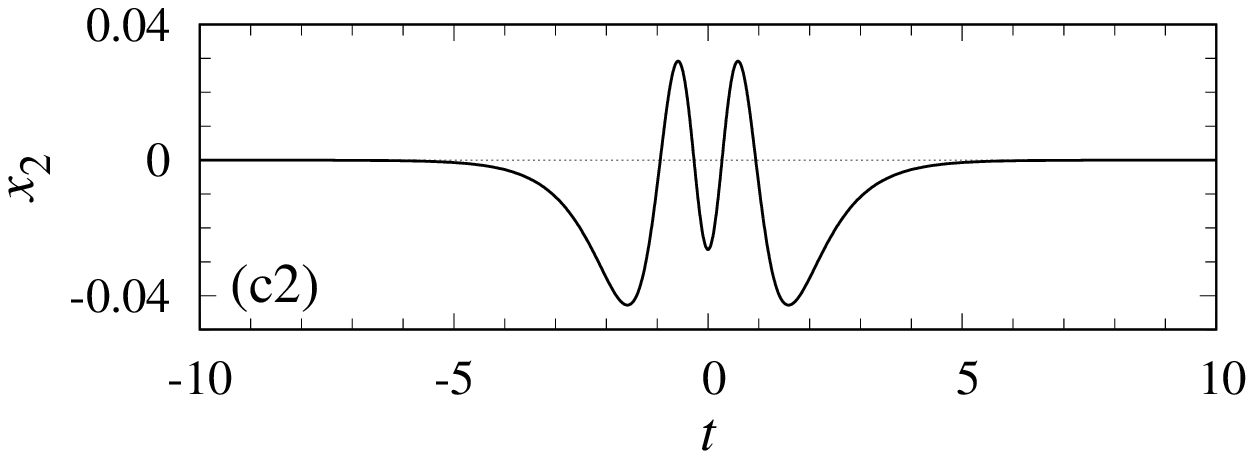}\
\includegraphics[scale=0.48]{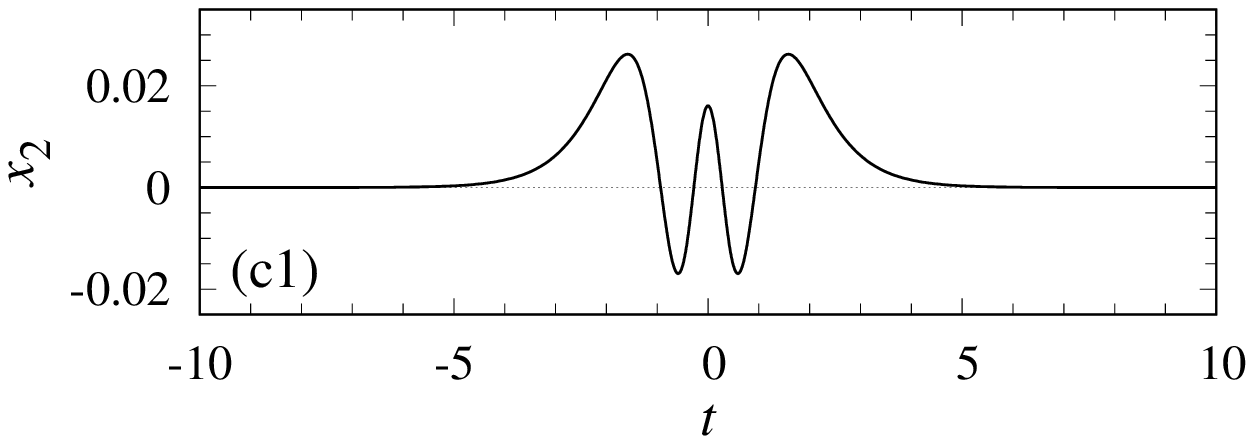}
\caption{Profiles of symmetric homoclinic orbits on the branches for $s=2$ and $\beta_3=4$:
(a1) and (a2) $\beta_2=-0.1$ and $\ell=0$; 
(b1) and (b2) $\beta_2=-0.05$ and $\ell=1$;
(c1) and (c2) $\beta_2=-0.006$ and $\ell=2$.
 \label{fig:4b}}
\end{center}
\end{figure}

\begin{figure}[t]
\begin{center}
\includegraphics[scale=0.53]{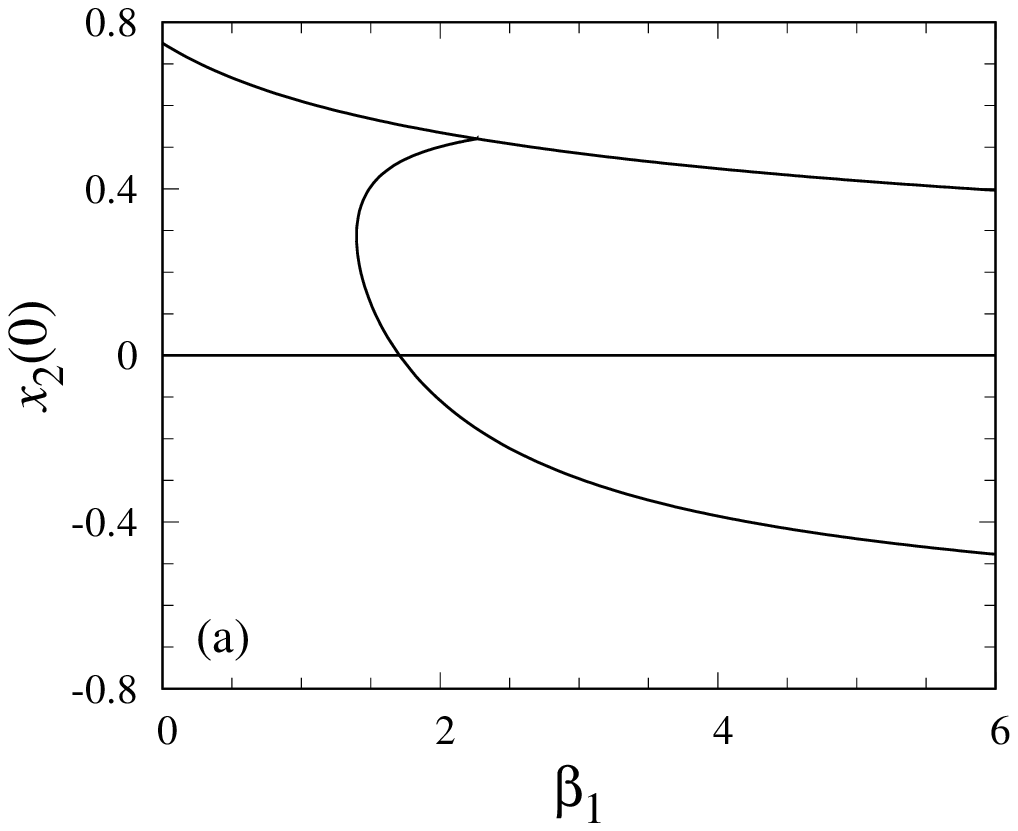}\quad
\includegraphics[scale=0.53]{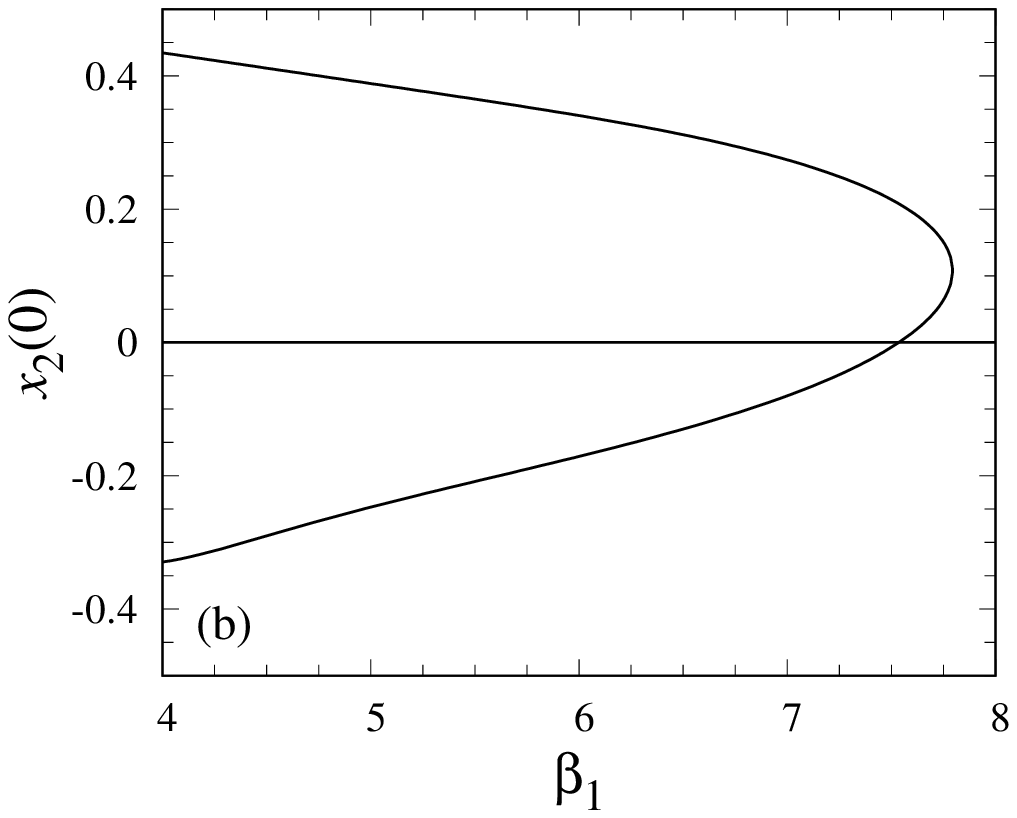}\\
\includegraphics[scale=0.53]{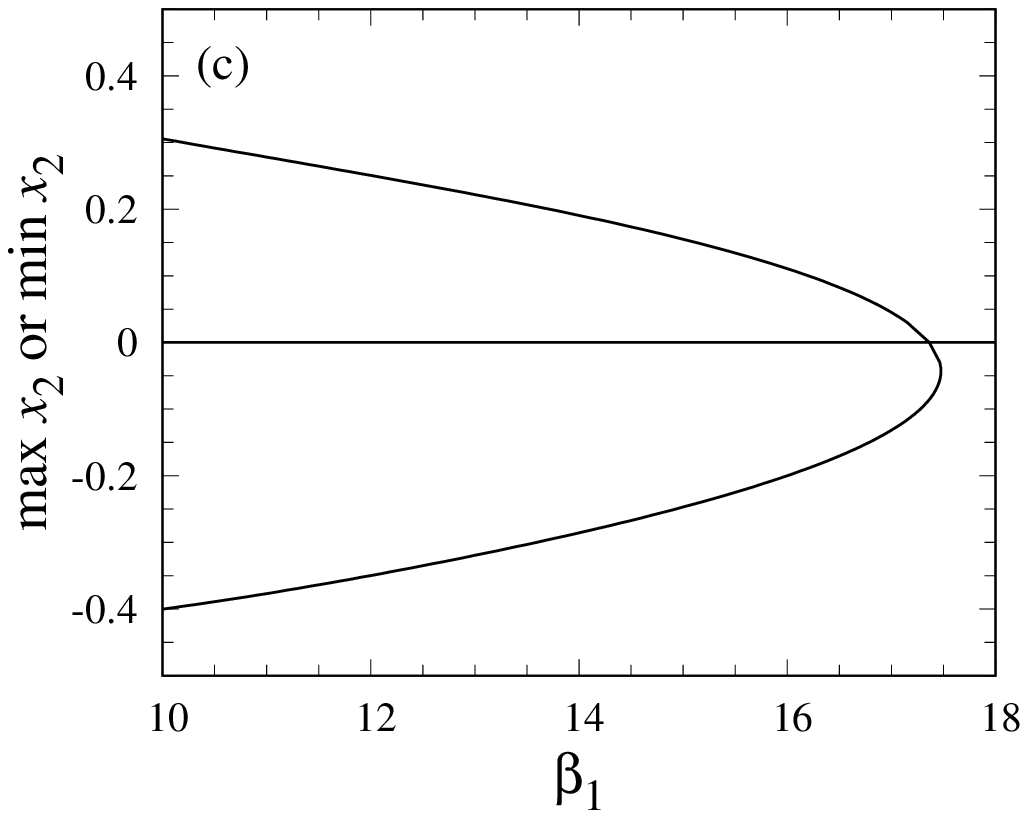}
\caption{Bifurcation diagrams for $s=2$, $\beta_2=0$ and $\beta_3=4$:
(a) $\ell=0$; (b) $\ell=1$; (c) $\ell=2$.
Here $\beta_1$ is taken as a control parameter.
\label{fig:4c}}
\end{center}
\end{figure}
\begin{figure}[t]
\begin{center}
\includegraphics[scale=0.48]{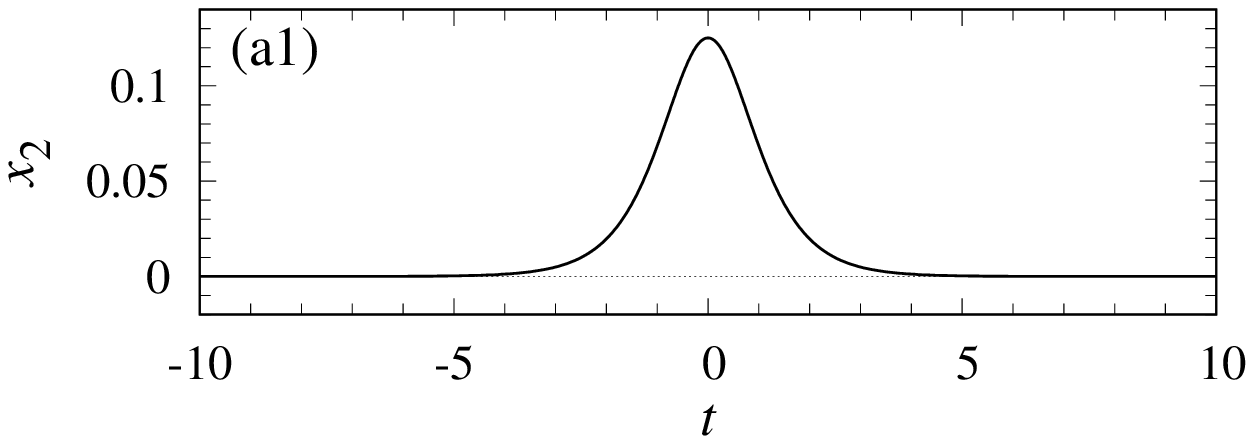}\
\includegraphics[scale=0.48]{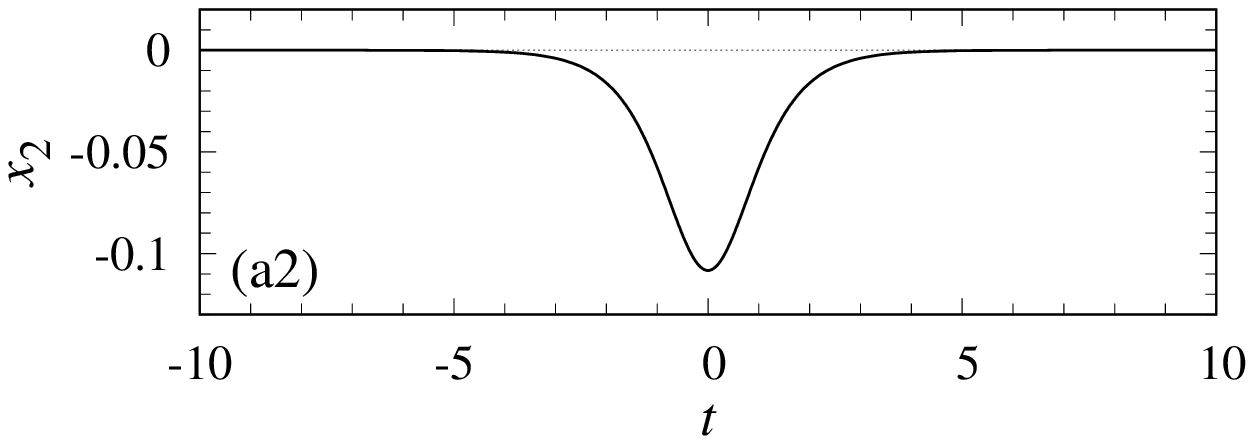}\\[1ex]
\includegraphics[scale=0.48]{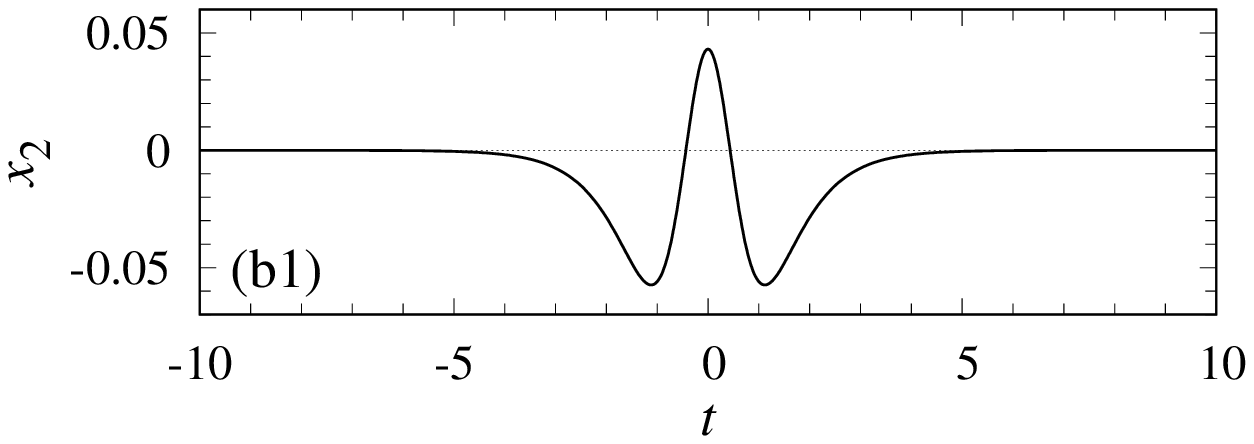}\
\includegraphics[scale=0.48]{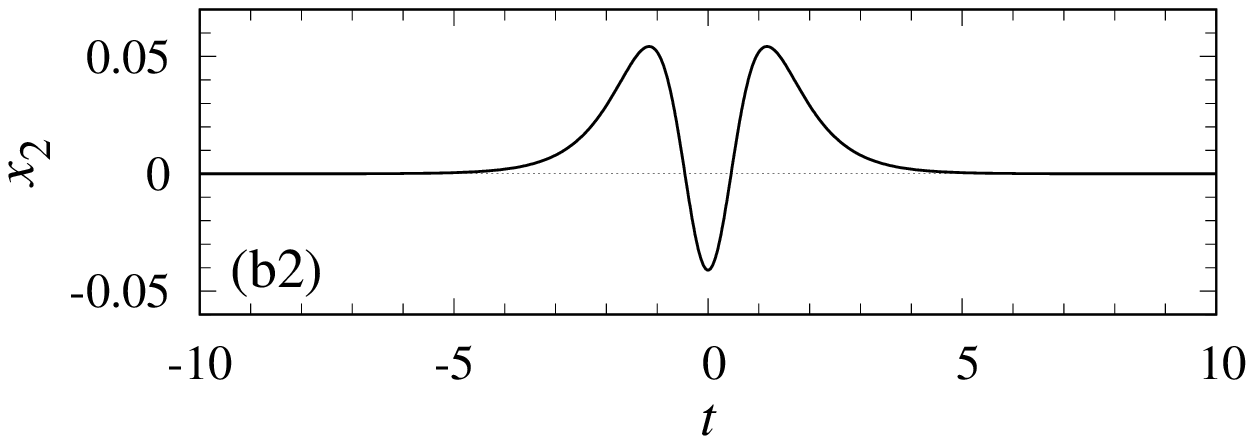}\\[1ex]
\includegraphics[scale=0.48]{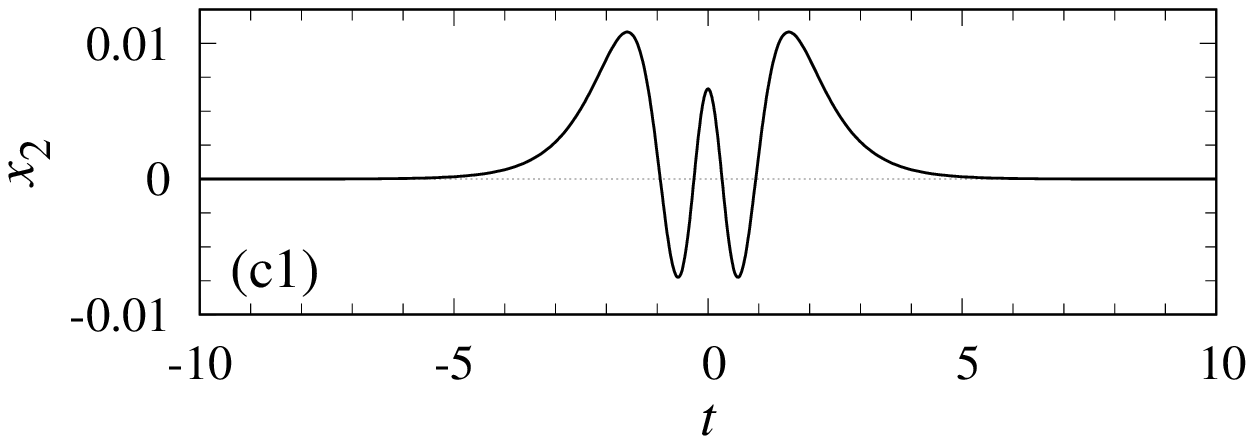}\
\includegraphics[scale=0.48]{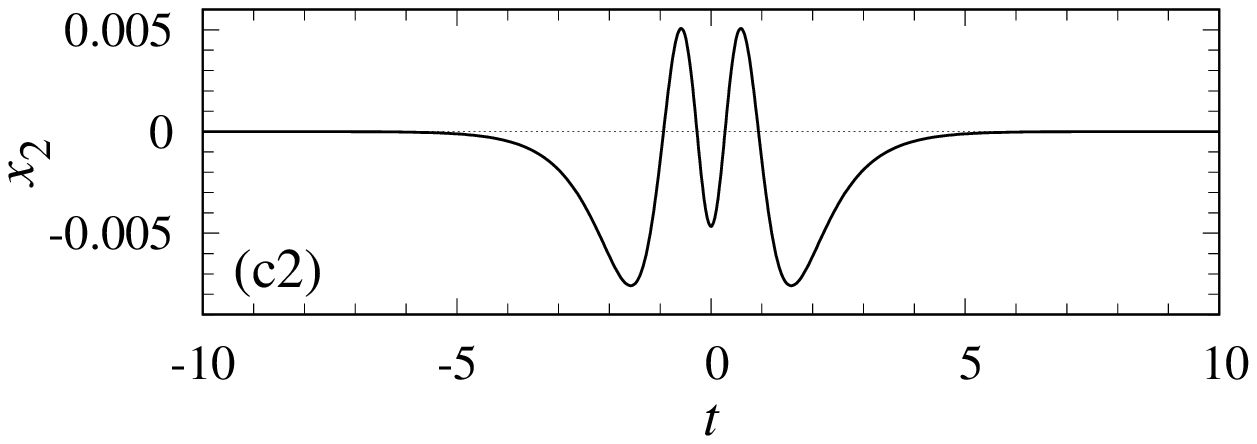}
\caption{Profiles of symmetric homoclinic orbits on the branches
 for $s=2$, $\beta_2=0$ and $\beta_3=4$:
(a1) $\beta_1=1.5$ and $\ell=0$; (a2) $\beta_1=2$ and $\ell=0$; 
(b1) $\beta_1=7.7$ and $\ell=1$; (b2) $\beta_1=7.3$ and $\ell=1$; 
(c1) $\beta_1=17.3$ and $\ell=2$; (c2) $\beta_1=17.4$ and $\ell=2$.
 \label{fig:4d}}
\end{center}
\end{figure}

\begin{figure}[t]
\begin{center}
\includegraphics[scale=0.53]{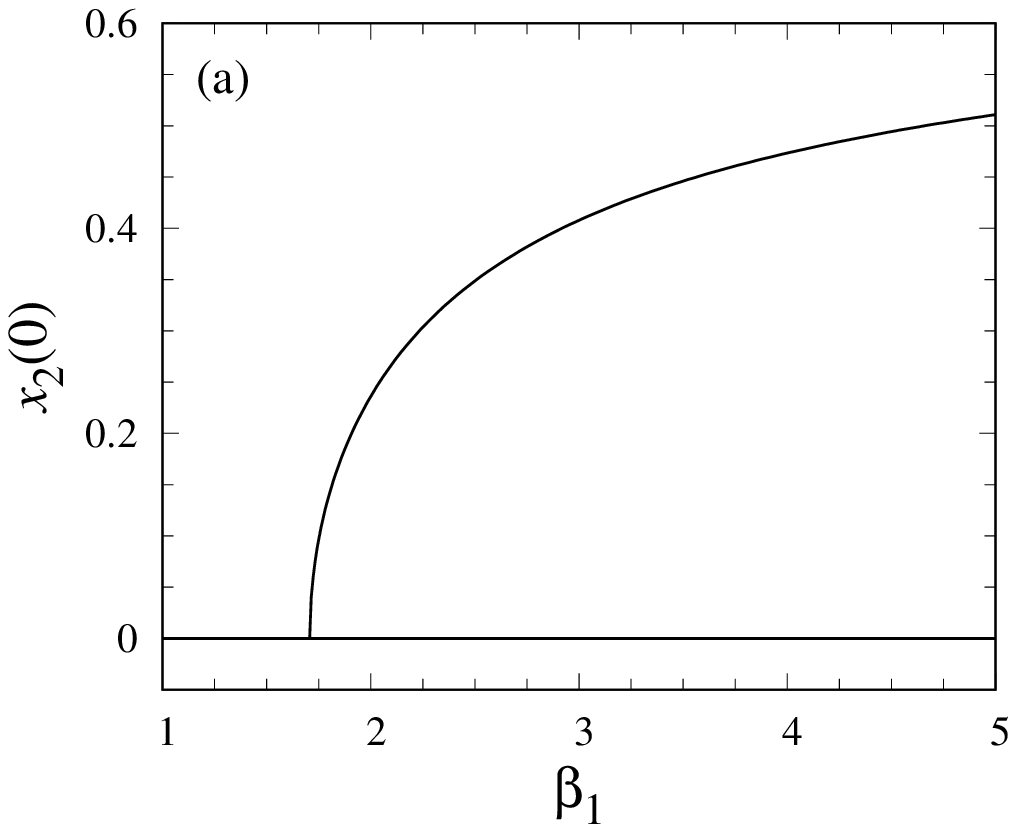}\quad
\includegraphics[scale=0.53]{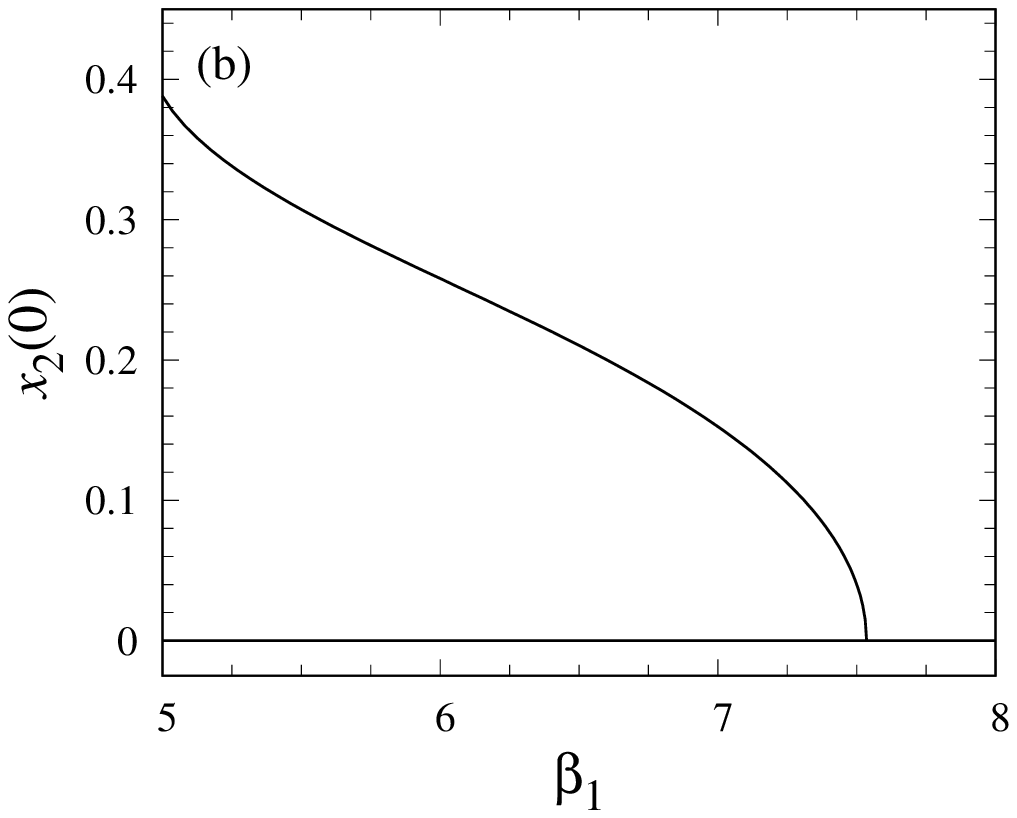}\\
\includegraphics[scale=0.53]{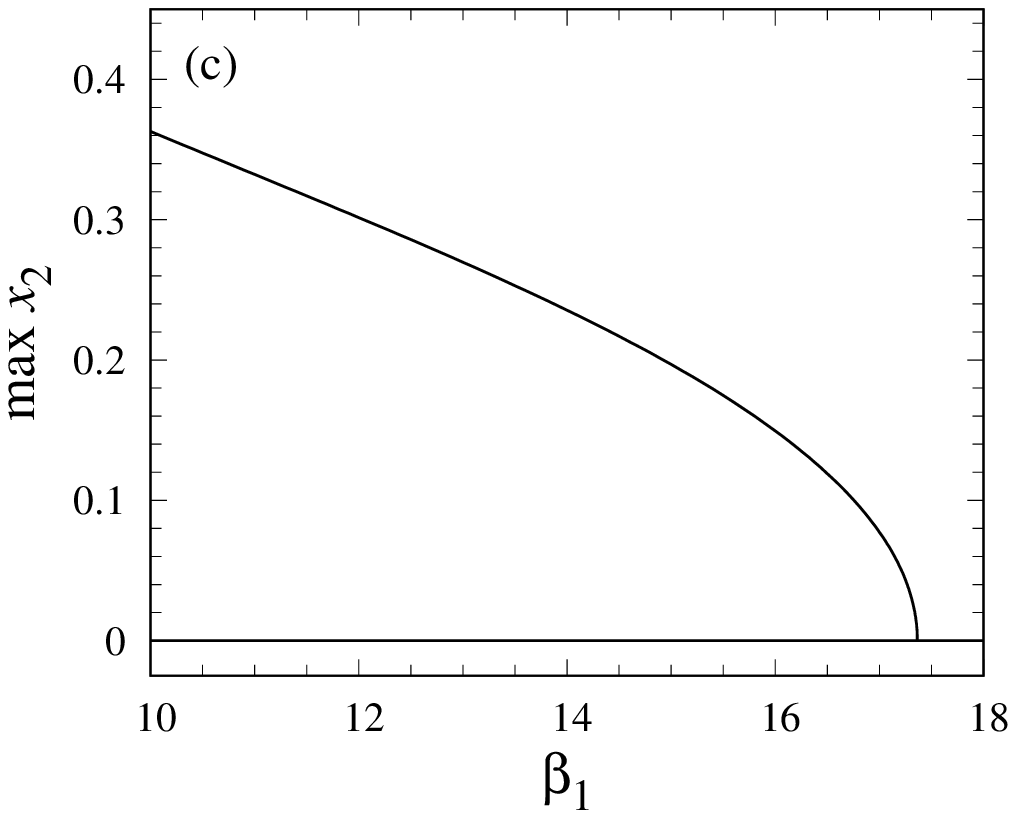}
\caption{Bifurcation diagrams for $s=2$ and $\beta_2,\beta_3=0$:
(a) $\ell=0$; (b) $\ell=1$; (c) $\ell=2$.
Here $\beta_1$ is taken as a control parameter.
\label{fig:4e}}
\end{center}
\end{figure}
\begin{figure}[t]
\begin{center}
\includegraphics[scale=0.48]{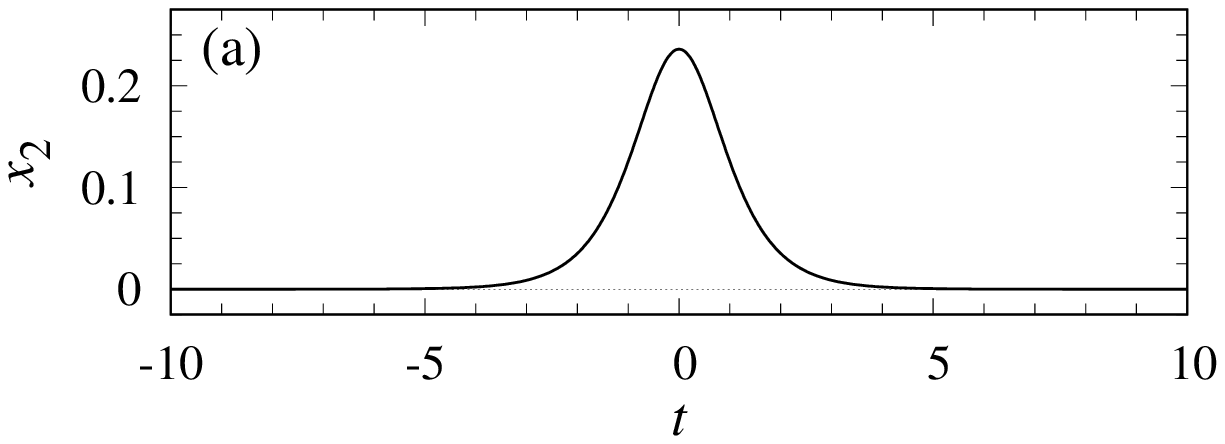}\
\includegraphics[scale=0.48]{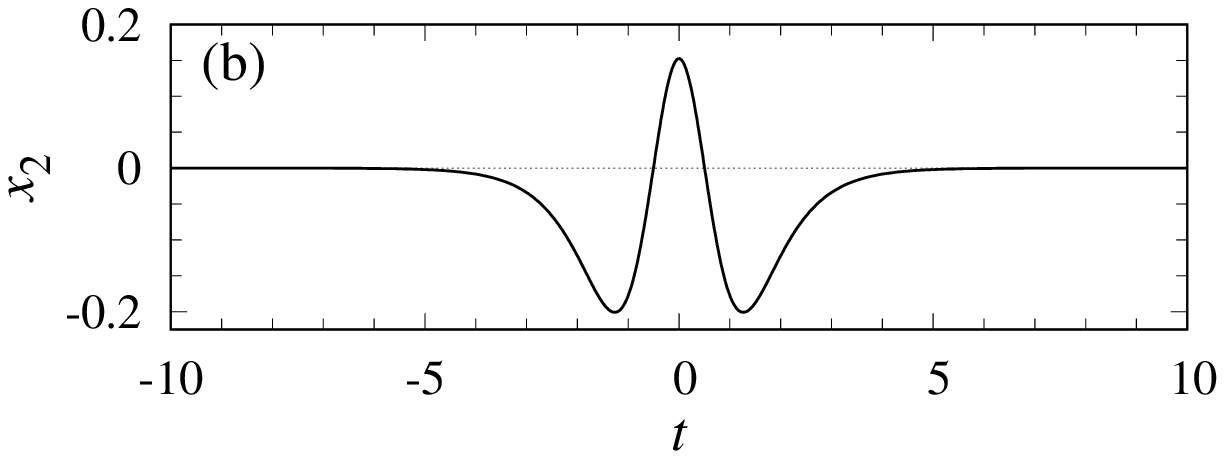}\\[1ex]
\includegraphics[scale=0.48]{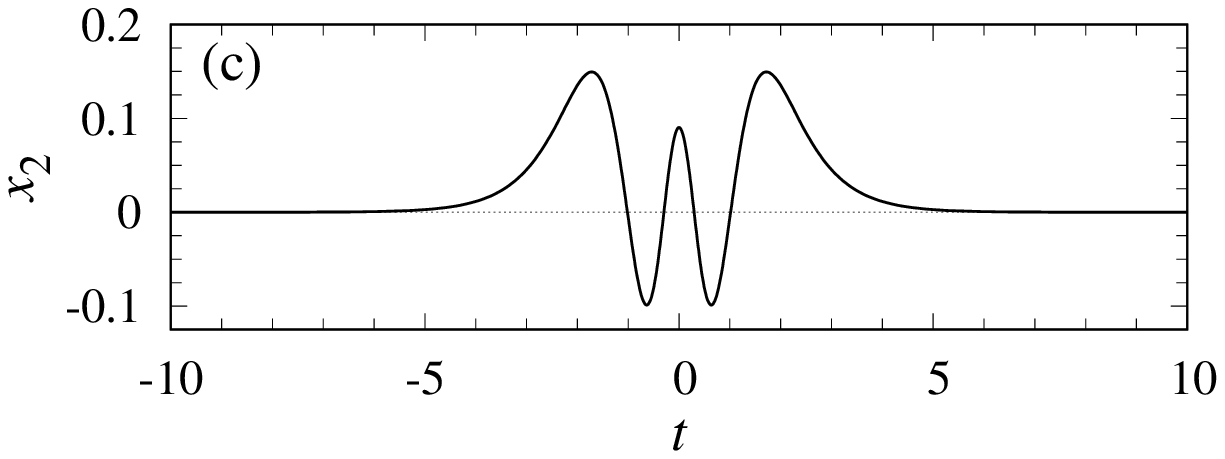}\
\caption{Profiles of symmetric homoclinic orbits on the branches
 for $s=2$ and $\beta_2,\beta_3=0$:
(a) $\beta_1=2$ and $\ell=2$; (b) $\beta_1=7$ and $\ell=1$; (c) $\beta_1=16$ and $\ell=2$.
\label{fig:4f}}
\end{center}
\end{figure}

Finally we give numerical computations for \eqref{eqn:ex}.
We take $s=2$ so that Eq.~\eqref{eqn:excon} gives $\beta_1=1.70710678\ldots$ for $\ell=0$,
 $\beta_1=7.5355339\ldots$ for $\ell=1$ and $\beta_1=17.36396103\ldots$ for $\ell=2$
 as the value of $\beta_1$ for which assumption~(R5) holds with $n_0=2$.
To numerically compute symmetric homoclinic orbits,
 we used the computer tool \texttt{AUTO} \cite{DO07}
 to solve the bondary value problem of \eqref{eqn:ex} with the boundary conditions
 \[
L_\s x(-T)=0,\quad
x(0)\in\fix(R),
 \]
where $T=20$ and $L_\s$ is the $2\times 4$ matrix
 consisting of two row eigenvectors with negative eigenvalues
 for the Jacobian matrix of \eqref{eqn:ex} at the origin,
\[
\begin{pmatrix}
 0 & 0 & 1 & 0\\
 0 & 0 & 0 & 1\\
1 & -\beta_2 & 0 & 0\\
-\beta_2 & s & 0 & 0
\end{pmatrix}.
\]

Figure~\ref{fig:4a} shows bifurcation diagrams for $\beta_3=4$
 when $\beta_1$ is fixed and satisfies \eqref{eqn:excon} for $\ell=0,1,2$
 and $\beta_2$ is taken as a control parameter.
In Fig.~\ref{fig:4a}(c) the maximum and minimum of the $x_2$-component
 are plotted as the ordinate when $x_2(0)$ is positive and negative, respectively.
We observe that a saddle-node bifurcation occurs at $\beta_2=0$
 while another saddle-node bifurcation occurs at a different value of $\beta_2$.
The $x_2$-components of symmetric homoclinic orbits born at the bifurcation point $\beta_2=0$
 in Fig.~\ref{fig:4a} are plotted in Fig.~\ref{fig:4b}.
We also see that they have $\ell+1$ extreme points
 like the corresponding bounded solutions to \eqref{eqn:exve2a}
 when $\beta_1$ satisfies \eqref{eqn:excon} with $\ell=0,1,2$.

Figure~\ref{fig:4c} shows bifurcation diagrams for $\beta_2=0$ and $\beta_3=4$
 when $\beta_1$ is taken as a control parameter.
Note that there exists a branch of $x_2(= x_4)=0$ for all values of $\beta_1$.
We observe that a transcritical bifurcation occurs at $\beta_1=0$ satisfying \eqref{eqn:excon} for $\ell=0,1,2$
 while another bifurcation occurs at a value of $\beta_1$ in Fig.~\ref{fig:4c}(a):
 Eq.~\eqref{eqn:ex} is $\Zset_2$-equivariant with the involution
\[
S': (x_1,x_2,x_3,x_4)\mapsto(-x_1,x_2,-x_3,x_4)
\]
and has a symmetric homoclinic orbit with $(x_1,x_3)=(0,0)$ for $\beta_2=0$,
 and a pitchfork bifurcation at which a pair of symmetric homoclinic orbits with $(x_1,x_3)\neq(0,0)$
 are born occurs there.
The $x_2$-components of symmetric homoclinic orbits born at the bifurcation points in Fig.~\ref{fig:4c}
 are plotted in Fig.~\ref{fig:4d}.

Figure~\ref{fig:4e} shows bifurcation diagrams for $\beta_2,\beta_3=0$
 when $\beta_1$ is taken as a control parameter.
Note that there exist a branch of $x_2(= x_4)=0$ for all values of $\beta_1$,
 and a pair of branches of solutions which are symmetric about $x_2=0$.
We observe that a pitchfork bifurcation occurs
 at values of $\beta_1$ satisfying \eqref{eqn:excon} for $\ell=0,1,2$.
The $x_2$-components of symmetric homoclinic orbits born at the bifurcation in Fig.~\ref{fig:4e}
 are also plotted in Fig.~\ref{fig:4f}.

\section*{Acknowledgements}
This work was partially supported by the JSPS KAKENHI Grant Numbers 25400168, 17H02859.

\end{document}